\documentclass[12pt,reqno]{amsart}

\usepackage{fullpage}

\usepackage{amssymb,amsfonts}
\usepackage[all,arc]{xy}
\usepackage{enumerate}
\usepackage{mathrsfs}
\usepackage{color}   
\usepackage{hyperref}
\hypersetup{
    colorlinks=true, 
    linktoc=all,     
    linkcolor=blue,  
}
\usepackage{amsthm}
\usepackage{amsmath}
\usepackage{graphicx}
\usepackage{wasysym}
\usepackage{pgf,tikz}
\usetikzlibrary{positioning}

\newcommand{\C}{\mathbf{C}}
\newcommand{\LJ}{\mathbf{LJ}}
\newcommand{\DS}{\mathrm{DS}}
\newcommand{\scc}{\mathscr{C}}
\newcommand{\scd}{\mathscr{D}}
\newcommand{\scr}{\mathscr{R}}

\newcommand{\Irr}{\mathrm{Irr}}

\newcommand{\Speh}{\mathrm{Speh}}

\newcommand{\bC}{\mathbb C}

\newcommand{\bn}{\mathbb N}

\newcommand{\bR}{\mathbb R}
\newcommand{\ba}{\mathbb A}

\newcommand{\ra}{\rangle}

\newcommand{\bs}{\backslash}

\newcommand{\al}{\alpha}

\DeclareMathOperator{\ind}{ind}
\DeclareMathOperator{\GL}{GL}
\DeclareMathOperator{\Hom}{Hom}
\DeclareMathOperator{\Ind}{Ind}

\DeclareMathOperator{\SL}{SL}

\newcommand{\fg}{\mathfrak g}

\newcommand{\sco}{\mathcal{O}}

\newcommand{\scu}{\mathcal{U}}

\newtheorem{Thm}{Theorem}[section]
\newtheorem{Prop}[Thm]{Proposition}
\newtheorem{Lem}[Thm]{Lemma}
\newtheorem{Cor}[Thm]{Corollary}

\newtheorem{Hypo}[Thm]{Hypothesis}

\theoremstyle{definition}
\newtheorem{Def}[Thm]{Definition}

\theoremstyle{remark}
\newtheorem{Rem}[Thm]{Remark}

\theoremstyle{definition}

\title{Quaternionic Speh representations}
\author{Yuanqing Cai}
\address{Faculty of Mathematics and Physics, Institute of Science and Engineering, Kanazawa University, Kakumamachi, Kanazawa, Ishikawa, 920-1192, Japan}
\email{cai@se.kanazawa-u.ac.jp}

\begin{document}

\maketitle

\begin{abstract}
For a central division algebra $D$, we study a family of representations of $\GL_{k,D}$ (both locally and globally), which can be viewed as analogues of the Speh representations. Locally, we study unique models for these representations. Globally, we show that these representations support certain non-vanishing Fourier coefficients. 
\end{abstract}

\tableofcontents



\section{Introduction}

The uniqueness of Whittaker models is a fundamental result in the study of automorphic representations and has many important applications. For example, it leads to the proof of the functional equations of certain automorphic $L$-functions via the Langlands-Shahidi method and several Rankin-Selberg integrals. Unfortunately, this important property does not hold for non-quasi-split groups. As a result, it seems rather difficult to develop a theory of $L$-functions for these groups. The purpose of this paper is to study a family of representations with unique models of $\GL_{k,D}$ (both locally and globally) for a central division algebra $D$ over a local or global field $F$. 

We first start with a simple example to get some basic ideas of this problem. Let $D$ be a central division algebra over a local field $F$ of dimension $d^2$. Then the only nilpotent orbit of the group $D^{\times}$ is the trivial orbit. As a consequence, only one-dimensional representations of $D^{\times}$ have unique models and most of the representations do not support unique models.

Similarly, for $\GL_{k,D}$, it is not difficult to see that most representations do not have unique models, either. In this article, we would like to search for representations of $\GL_{k,D}$ with unique models. 

A natural method of constructing representations of $\GL_{k,D}$ is the Jacquet-Langlands correspondence. The Jacquet-Langlands correspondence (as in \cite{DKV84}) for discrete series representations says that there is a bijection between discrete series representations of $\GL_{k,D}$ and $\GL_{kd}$ satisfying a character identity. This correspondence was later extended in \cite{Badulescu08, BR10} to allow unitary representations as well. For our purpose, we would like to take the latter version for the following two reasons. First, it does produce more representations than the discrete series version. Second, it is also local-to-global compatible and will be more suitable if we are aiming for some global applications. 

For an admissible representation $\pi$ of a $p$-adic group, a result of M{\oe}glin--Waldspurger \cite{MW89} says that the dimension of a certain degenerate Whittaker models for $\pi$ is related to the leading coefficients of the local character expansion of $\pi$. As a result, one hope that the Jacquet-Langlands transfer of appropriate representations of $\GL_{kd}$ (with suitable unique degenerate Whittaker models) might give some desired representations of $\GL_{k,D}$ as the Jacquet-Langlands transfer satisfies a character identity. At least in the non-Archimedean case, this idea works for the Speh representations of $\GL_{kd}$. 

The Speh representations were first studied in the real case and later this construction was also extended to the $p$-adic case. Locally, we view the construction of the Speh representations as the following: 
\[
\tau \mapsto \Speh(\tau,n)
\]
where $\tau$ is an irreducible generic unitary representation of $\GL_{k}$ and $\Speh(\tau,n)$ is the ``smallest'' piece of a highly reducible induced representation defined using $\tau$. A key property of the Speh representations is that the maximal orbit that supports degenerate Whittaker models is $(k^n)$. Moreover, the dimension of such degenerate Whittaker models is $1$. 

Note that if $n=dn'$ for some $n'$, then the nilpotent orbit $(k^n)$ is an orbit that ``comes from'' $\GL_{kn',D}$. Our attempt is to define 
\[
\Speh_D(\tau,n):=|\LJ|(\Speh(\tau,nd)). 
\]
Here $|\LJ|$ is the Jacquet-Langlands transfer for unitary representations in \cite{Badulescu08,BR10}.  This definition works both locally and globally. 

Our main theorem is the following. 

\begin{Thm}\label{thm:main}
We have the following: 
\begin{enumerate}
\item (Local vanishing result) For any nilpotent orbit $\mathcal{O}$ greater than or not comparable to $(k^n)_D$, the representation does not support degenerate Whittaker models for $\mathcal{O}$.  \item (Local multiplicity one result) The Archimedean case is based on certain natural hypothesis to be verified. Then the dimension 
\[
\dim \Hom_{N_{(k^n)_D}}(\Speh_D(\tau,n),\psi_{(k^n)_D})\leq 1.
\]
Here, the pair $(N_{(k^n)_D}, \psi_{(k^n)_D})$ is the unipotent subgroup and character used in the definition of the degenerate Whittaker models for the nilpotent orbit $(k^n)_D$. Moreover, in the non-Archimedean case, this dimension is exactly $1$. 
\item (Global result)
The maximal nilpotent orbit that supports nonzero global degenerate Whittaker coefficient $\Speh_D(\tau,n)$ is $(k^n)_D$. 
\end{enumerate}
\end{Thm}

In the real case, the Speh representations originally refer to $\Speh_{\mathbb{R}}(\tau,2)$ where $\tau$ is a discrete series representation of $\GL_2(\mathbb{R})$. The representation $\Speh_{\mathbb{H}}(\tau,1)$ can be viewed as quaternionic analogues of the Speh representations. We call them \textit{quaternionic Speh representations}.

Let us now say a few words regarding the proofs. As we indicated above, the non-Archimedean case can be done using the character identity without too much trouble. The more difficult part is deal with the Archimedean and global theory. 

The Archimedean version of the result of M{\oe}glin-Waldspurger is not known at the moment. Partial results can be found in \cite{GS15, GGS17} and these allows us to settle the vanishing part. For the multiplicity one part, we use the following idea. First, the definition of Speh representations is given by induction, thus one first consider the case when $\tau$ is a discrete series. For this case, we use a method in \cite{KP84}. In this paper, Kazhdan--Patterson used a global method to prove that the local components of the theta representations at bad primes support unique Whittaker models if the same holds for unramified places. This method can be adapted to our case, under certain natural hypothesis on the Kirillov models for representations of $\GL_{k,D}$. This only treats representations that can be realized as local components of global representations, but should be sufficient for applications. 

Hang Xue suggested to us that in the minimal case, $\tau \mapsto \Speh_D(\tau,1)$ can be realized using the theta correspondence. Thus, in this particular case, the desired unique model is a consequence of a result by Gomez-Zhu \cite{GZ14}. We will explain this in Section \ref{sec:theta}. 

The global correspondence is proved using the method of the trace formula. As a result, it seems difficult to gain information regarding the Fourier coefficients. In \cite{KP84}, another method was used to show that theta representations in certain cases are globally generic. This is again adaptable to our case to prove a base case, and we prove the general case using an induction-in-stages argument. 

We end this introduction by saying a few words regarding some potential applications of our results. 
The first application is the twisted doubling integrals \cite{CFGK19}. The twisted doubling method is a generalization of the doubling method. It gives a family of Rankin-Selberg integrals that represents the tensor product $L$-function $L(s,\pi\times \tau)$ for $\pi$ of a classical group, and $\tau$ of a general linear group $\GL_k$. A key ingredient in the construction is the use of the generalized Speh representations, which can be viewed as
\[
\tau\in \mathrm{DS}_{\mathrm{cusp}}(\GL_k(\mathbb{A})) \mapsto \Speh(\tau,n) \in \mathrm{DS}(\GL_{kn}(\mathbb{A}))
\]
for every positive integer $n$. Here, $\mathrm{DS}(\GL_k(\mathbb{A})) $ denotes the set of discrete series representations of $\GL_{k}(\mathbb{A})$ and the subscript $\mathrm{cusp}$ indicated cuspidal representations. The unfolding argument and the Eulerian property rely on the fact that the representation $\Speh(\tau,n)$ is supported on a sufficiently small nilpotent orbit and admits unique models of degenerate type at every local place. To extend the twisted doubling integrals to the quaternion unitary groups (see \cite{Cai21}), we seek an analogous construction
\[
\tau\in \mathrm{DS}_{\mathrm{cusp}}(\GL_k(\mathbb{A})) \mapsto \Speh_D(\tau,n) \in \mathrm{DS}(\GL_{kn,D}(\mathbb{A}))
\]
for a positive integer $n$ and a central division algebra $D$ over $F$. We also have to prove analogous properties for these representations. This is what we are seeking in this paper.  

Another application is related to Lapid--Mao \cite{LM20}. In this paper, a local version of the Rankin-Selberg convolution of two Speh representations is given and several properties are studied. It is also mentioned that a global version is possible by convolving two global Speh representations modulo some regularization problem. We would like to suggest that, one can take convolution of two representations of the form $\Speh_D(\tau,1)$ and its local integrals will also be the ones studied by Lapid--Mao. It is possible for $\Speh_D(\tau,1)$ to be cuspidal so that no regularization is necessary. 

The rest of the paper is organized as follows. In Section \ref{sec:preliminaries} we recall some preliminary results. In particular, wereview the classification of unitary representations of $\GL_{k,D}$. We review the extended Jacquet-Langlands correspondence, following \cite{Badulescu08} and \cite{BR10} in Section \ref{sec:local JL}. In Section \ref{sec:quaternionic speh}, we define the Speh representations over central division algebras locally and study some properties. In particular, the non-Archimedean part of Theorem \ref{thm:main} is proved and Section \ref{sec:Arch I} treats the vanishing part in the Archimedean case. We start the global investigation in Section \ref{sec:global Speh}. Section \ref{sec:global non-vanishing} proves the global non-vanishing statement. Moreover, the uniqueness part in the Archimedean case is proved using global method in Section \ref{sec:Archimedean using global}. In Appendix \ref{app:Kirillov}, we prove a result related to Kirllov models for $\GL_{k,D}$ in the non-Archimedean case. 

\subsection*{Acknowledgement}
The author would like to thank Hengfei Lu, Hang Xue and Lei Zhang for helpful discussions. This work was supported by MEXT Leading Initiative for Excellent Young Researchers Grant Number JPMXS0320200394.

\section{Preliminaries}\label{sec:preliminaries}

Let $F$ be a local field of characteristic zero. Let $D$ be a central division algebra over $F$ of dimension $d^2$. For a positive integer $k$, set $G_k=\GL_k(F)$ and $G_k'=G_{k,D}=\GL_k(D)$. From now on we identify a smooth representation of finite length with its equivalence class, so we will consider two equivalent representations as being equal. 

We introduce the following notation:
\begin{itemize}
\item For an admissible representation $\pi$, we denote $\chi_{\pi}$ the function character of $\pi$. This is a map which is stable under conjugation and defined on the set of regular semisimple elements of $G_k$.  
\end{itemize}

For all positive integers $k$, we fix the following notations:
\begin{itemize}
\item $\Irr_{k,D}$: the set of irreducible smooth representations of $G_{k,D}$
\item $\scd_{k,D}$: the subset of essentially square integrable representations in $\Irr_{k,D}$
\item $\scc_{k,D}$: the subset of cuspidal representations in $\scd_{k,D}$
\item $\Irr_{k,D}^u$ (resp. $\scd_{k,D}^u$, $\scc_{k,D}^u$): the subset of unitary representations in $\Irr_{k,D}$ (resp. $\scd_{k,D}$, $\scc_{k,D}$)
\item $\Irr_{k,D}^{eu}$: the subset of essentially unitary representations in $\Irr_{k,D}$
\item $\scr_{k,D}$: the Grothendieck groups of admissible representations of finite length of $G_{k,D}$
\item $\nu=\nu_{k,D}$: the character of $G_{k,D}$, defined by the absolute value of the reduced norm of the determinant 
\item $\times$: the standard notation for normalized parabolic induction; see also \cite{BZ77}.
\end{itemize}

Moreover, all the induced representations are normalized. 



Let $\mathfrak{g}_{k,D}$ be the Lie algebra of $G_{k,D}$. The coadjoint nilpotent orbits in $\mathfrak{g}_{k,D}^{\ast}$ are classified by partitions of $k$. A typical orbit is denoted by $(k_1^{n_1} \cdots k_m^{n_m})_D$ with $k_1n_1 + \cdots + k_mn_m=k$.

When $D=F$, we usually drop the subscript in the above notations.  For example, $(1^k)$ is the trivial orbit in $\mathfrak{g}_{k}^{\ast}$. 

We introduce the following observation which may be useful when considering local-to-global questions. Let $D'=\mathrm{M}_n(D)$. Then $G_{k,D'}\simeq G_{kn,D}$. We can fix such an isomorphism and the nilpotent orbit $(k_1\cdots k_m)_{D'}$ corresponds to the nilpotent orbit $(k_1^n \cdots k_m^n)_{D}$. 

More generally, for nilpotent orbits $\mathcal{O}=(k_1^d\cdots k_m^d)$ of $\GL_{kd}$ and $\mathcal{O}'=(k_1\cdots k_m)_D$ of $\GL_{k,D}$, we say that they correspond to each other and write $\mathcal{O} \leftrightarrow \mathcal{O}'$. 

\subsection{Unitary dual}

In this section, we review the classification of the unitary dual of $G_{k,D}$ (which also includes the case of $G_k$). 

Let $\sigma \in \mathscr{D}_{l,D}$ (meaning: essentially square-integrable modulo center). Consider $\sigma\times \nu^{\alpha}\sigma$ with $\alpha>0$. There exists a smallest number $\alpha_0>0$ such that $\sigma\times \nu^{\alpha_0}\sigma$ is reducible. 

\begin{Def}
Let $\sigma\in\mathscr{D}_{l,D}$. Set $\nu_{\sigma}=\nu^{\alpha_0}$, where $\alpha_0$ is the smallest real number $\alpha>0$ such that $\sigma \times \nu^{\alpha}\sigma$ is reducible. 
\end{Def}

For $\sigma \in \mathscr{D}_{l,D}$ and a positive integer $n$, we define $u(\sigma,n)$ to be the Langlands quotient of 
\[
\nu_{\sigma}^{(n-1)/2} \sigma \times \nu_{\sigma}^{(n-3)/2} \sigma \times \cdots \times \nu_{\sigma}^{(1-n)/2}\sigma.
\]
The representation $u(\sigma,n)$ is an irreducible representation of $G_{ln,D}$. 

For $\sigma \in \mathscr{D}_{l,D}$, a positive integer $n$ and a real number $\alpha$, we denote by $\pi(u(\sigma,n),\al)$ the induced representation 
\[
\nu_{\sigma}^\al u(\sigma,n) \times \nu_{\sigma}^{-\al} u(\sigma,n).
\]
The representation $\pi(u(\sigma,n),\al)$ is also irreducible. 

The unitary dual of $G_{k,D}$ is given as follows. 

\begin{Thm}
Let $\scu_D$ be the set of all representations of type $u(\sigma,n)$ or $\pi(u(\sigma,n),\al)$ where $l,n$ range over all positive integers, $\sigma\in \scd_{l,D}^u$ and $\alpha\in (0,1/2)$. Then we have the following:
\begin{enumerate}
\item all the representations in $\scu_D$ are unitary;
\item any product of representations in $\scu_D$ is irreducible and unitary;
\item every irreducible unitary representation $\pi$ of $G_{k,D}$ is a product of representations in $\scu_{D}$.
\end{enumerate}
\end{Thm}

We refer the reader to \cite{BR10} Section 7 for a comphensive history of this result. 

\subsection{Classification of the generic unitary dual}

The generic unitary dual of $G_{k}$ is given as follows. 

\begin{Thm}
Let $\scu_{gen}$ be the set of all representations of type $u(\sigma,1)$ or $\pi(u(\sigma,1),\al)$ where $l$ range over all positive integers, $\sigma\in \scd_l^u$ and $\alpha\in (0,1/2)$. Then we have the following:
\begin{enumerate}
\item all the representations in $\scu_{gen}$ are unitary and generic;
\item any product of representations in $\scu_{gen}$ is irreducible generic and unitary;
\item every irreducible generic unitary representation of $G_{k}$ is a product of representations in $\scu_{gen}$.
\end{enumerate}
\end{Thm}

We refer the reader to \cite{BR10} Section 8 for more details. 

\subsection{The local Jacquet-Langlands correspondence}

Let $g\in G_{kd}$ and $g'\in G_{k,D}$. We say that $g$ \textit{corresponds} to $g'$ if both $g$ and $g'$ are regular semisimple and have the same characteristic polynomial. We shortly write $g\leftrightarrow g'$. Denote $G_{kd,d}$ the set of elements $g\in G_{kd}$ such that there exists $g'\in G_{k,D}$ with $g \leftrightarrow g'$.

The following theorem is proved in \cite{DKV84} if the characteristic of the base field $F$ is zero and \cite{Badulescu02} for the non-zero characteristic case.

\begin{Thm}
There is a unique bijection $\C:\scd_{kd}\to \scd_{k,D}$ such that for all $\sigma\in\scd_{kd}$ we have
\[
\chi_{\sigma}(g)=(-1)^{kd-k}\chi_{\C(\sigma)}(g')
\]
for all $g\in G_{kd}$ and $g'\in G_{k,D}$ such that $g\leftrightarrow g'$.
\end{Thm}

We identify the centers of $G_{kd}$ and $G_{k,D}$ via the canonical isomorphism. Thus the correspondence $\C$ preserves central characters. In particular, $\sigma\in \scd_{kd}^{u}$ if and only if $\C(\sigma)\in \scd_{k,D}^{u}$. 

\subsection{Classification of $\mathscr{D}_{k,D}$}

In this section, we review some necessary results on the classification of the discrete series representations. 

\subsubsection{The case $D=F$, non-Archimedean}
The classification of $\scd_k$ is given in terms of $\scc_l$, $l|k$. 

Let $l$ and $n$ be two positive integers and set $k=ln$. Let $\rho\in \scc_l$. Then the representation 
\[
\rho \times \nu\rho \times \cdots \times \nu^{n-1}\rho
\]
has a unique irreducible quotient $\sigma$. The representation $\sigma$ is an essentially square integrable representation of $G_k$. Notation: $\sigma=Z(\rho,n)$. Every $\sigma\in\scd_k$ is obtain in this way and $l,n$ and $\rho$ are determined by $\sigma$. (See \cite{Zelevinsky80} Section 9.) Moreover, for $\sigma \in \scd_k$, $\nu_{\sigma}=\nu$. 

We also define $Z^u(\rho,n)$ to be the unique irreducible quotient of 
\[
\nu^{(1-n)/2}\rho \times \nu^{(3-n)/2}\rho \times \cdots \times \nu^{(n-1)/2}\rho.
\]
This is a unitary representation. 

\subsubsection{The case of general $D$, non-Archimedean}

Let $l$ be a positive integer and $\rho'\in\scc_{l,D}$. Then $\sigma=\C^{-1}(\rho')$ is an essentially square integrable representation of $G_{dl}$. We may write $\sigma=Z(\rho,p)$ for some $p$ and $\rho\in\scc_{dl/p}$. Set $s(\rho')=p$ and $\nu_{\rho'}=\nu^{s(\rho')}$.

Let $n$ be a positive integers and set $k=ln$. Then the representation 
\[
\rho' \times \nu_{\rho'}\rho' \times \cdots \times \nu_{\rho'}^{n-1}\rho'
\] 
has a unique irreducible quotient $\sigma'$. The representation $\sigma'$ is an essentially square integrable representation of $G_{k,D}$. Notation: $\sigma'=T(\rho',n)$. Every $\sigma'\in\scd_{k,D}$ is obtained in this way and $l,n$ and $\rho'$ are determined by $\sigma'$. We then set $s(\sigma')=s(\rho')$. For this classification, see \cite{Tadic90}.

For $\sigma' \in \scd_{k,D}$, we have $\nu_{\sigma'}=\nu^{s(\sigma')}$. 

\subsubsection{Some notations}

Let $\sigma'\in \scd_{k,D}^{u}$. For any positive integer $n$, recall that $u(\sigma',n)$ the Langlands quotient of the induced representation  
\[
\nu_{\sigma'}^{(n-1)/2} \sigma' \times \nu_{\sigma'}^{(n-3)/2} \sigma' \times \cdots \times \nu_{\sigma'}^{(1-n)/2}\sigma'.
\] 
We denote by $\bar u(\sigma',n)$ the Langlands quotient of the induced representation 
\[
\nu^{(n-1)/2} \sigma' \times \nu^{(n-3)/2} \sigma' \times \cdots \times \nu^{(1-n)/2}\sigma'.
\] 
Both $u(\sigma',k)$ and $\bar u(\sigma',k)$ are irreducible representations. 

\subsubsection{The Archimedean case}

We refer the reader to \cite{BR10} for a comprehensive discussion. We first discuss the real case. In this case, $\mathscr{D}$ consists of the following representations: 
\begin{itemize}
\item the unitary characters of $\mathbb{R}^{\times}$;
\item the Langlands quotient of the induced representation 
\[
\chi\nu^{n} \times \chi\nu^{-n}
\]
for a positive integer $n$.
\end{itemize}
For all $\sigma \in \mathscr{D}$, $\nu_{\sigma}=\nu$. 



In the case $F=\mathbb{C}$, the set $\mathscr{D}_{\mathbb{C}}$  is the set of unitary characters of $\mathbb{C}^{\times}$. 

Finally, if $D=\mathbb{H}$ is the unique quaterion algebra over $\mathbb{R}$, then $\mathscr{D}_{\mathbb{H}}$ consists of all the irreducible finite-dimensional representations of $\mathbb{H}^{\times}$. Moreover, $\nu_{\sigma}=\nu^2$ if $\dim \sigma =1$; $\nu_{\sigma}=\nu$ if $\dim \sigma >1$. 

\subsection{The involution}
We need the Aubert involution in the non-Archimedean case. We will use the notation only and will not use any explicit calculation. 

Aubert defined in \cite{Aubert95} an involution 
of the Grothendieck group of smooth representations of finite length of a reductive group over a local non-Archimedean field. The involution sends an irreducible representation to an irreducible representation up to a sign. We specialize this involution to $G_k$ (resp. $G_{k,D}$) and denote it $i_k$ (resp. $i_k'$). We will write $i$ and $i'$ when the index is not relevant or it is clearly understood. With this notation we have the relation $i(\pi_1)\times i(\pi_2)=i(\pi_1\times \pi_2)$, i.e. “the involution commutes with the parabolic induction”. The same holds for $i'$. The reader may find all these facts in \cite{Aubert95}.

If $\pi \in \Irr_k$, then one and only one of $i(\pi)$ and $-i(\pi)$ is an irreducible representation. We denote it by $|i(\pi)|$. We denote $|i|$ the involution of $\Irr_k$ defined by $\pi\mapsto |i(\pi)|$.  The same facts and definitions also hold for $i'$.




\section{Local Jacquet-Langlands correspondence}\label{sec:local JL}

In order to define a global Jacquet-Langlands correspondence for automorphic representations, it is not sufficient to transfer only square integrable representations as in the classical theory (for example, see \cite{DKV84}). It would be necessary to transfer at least the local components of global discrete series. This was achieved in \cite{Badulescu08,BR10}. In these two papers, the local transfer for all unitary representations is established. A global correspondence for discrete series compatible with the local transfer is therefore proved.

In this section, we review the Jacquet-Langlands correspondence as proved in \cite{Badulescu08,BR10}. The notations are $\C$ (for discrete series representations), $\LJ$ (for the Grothendieck ring) and $|\LJ|$ (for all unitary representations), respectively. 

\subsection{The extended correspondence}

The correspondence $\C^{-1}$ can be extended in a natural way to a correspondence between the Grothendieck groups. 

\begin{Thm}[\cite{Badulescu07}]
\begin{enumerate}
\item For all positive integers $k$, $\LJ_k$ is the unique map from $\scr_{kd}$ to $\scr_{k,D}$ such that for all $\pi\in \scr_{kd}$, we have 
\begin{equation}\label{eq:sign in LJ for Grothendieck}
\chi_{\pi}(g) = (-1)^{kd-k} \chi_{\LJ_k(\pi)}(g')
\end{equation}
for all $g\leftrightarrow g'$. 
\item The maps $\LJ_k$ commute with the parabolic induction.  
\end{enumerate}
\end{Thm}

We say that $\pi\in \scr_{kd}$ is $d$-compatible if $\LJ_k(\pi)\neq 0$. This means that there exists $g\in G_{kd,d}$ such that $\chi_{\pi}(g)\neq 0$. Note that this definition depends only on $d$, not on $D$. 

In this paper, we only need to understand this correspondence for unitary representations. By the classification of unitary representations of $G_k$, it suffices to understand $\LJ$ for representations of the form $u(\sigma,n)$.

\begin{Rem}
The convention in \cite{BR10} Section 4 (the Archimedean case) is slightly different from \cite{Badulescu08}. To unify the presentation, we also take the convention in \cite{Badulescu08} in the Archimedean case.

Moreover, in the Archimedean case, the sign in \eqref{eq:sign in LJ for Grothendieck} is not explicitly given in \cite{BR10}. But it is not difficult to see that the calculation can be done as in the non-Archimedean case. 
\end{Rem}

\subsection{Transfer of $u(\sigma,n)$: non-Archimedean case}

We now collect results regarding the transfer of $u(\sigma,k)$ (see \cite{Badulescu08} Section 3.2). 

Let $n,l,q$ be positive integers. Set $k=lnq$. Let $\rho\in \scc_q^u$ and $\sigma=Z^u(\rho,l)\in \scd_{lq}^u$, $\tau=Z^u(\rho,n)\in \scd_{nq}^u$.

Let $s$ be the smallest positive integer such that $d\mid sq$. We first define the transfer of $u(\sigma,n)$. This question has no meaning unless $d\mid k$ (i.e. $s\mid ln$) which we shall assume.

\begin{Prop}[\cite{Badulescu08} Proposition 3.7]\label{prop:transfer of unitary non-arch}
\begin{enumerate}
\item If $d\mid lq$ (i.e. $s\mid l$), then $\sigma'\in \C(\sigma)$ is well-defined; we have $s=s(\sigma')$ and
\[
\LJ(u(\sigma,n))=\bar u(\sigma',n).
\]
\item If $d\mid nq$ (i.e. $s\mid n$), then $\tau'=\C(\tau)$ is well-defined; we have $s=s(\tau')$ and
    \[
    \LJ(u(\sigma,n))=\varepsilon |i'(\bar u(\tau',l))|
    \]
    where $\varepsilon=1$ if $s$ is odd and $\varepsilon=(-1)^{\frac{ln}{s}}$ if $s$ is even.
\item If $d$ does not divide neither $lq$, nor $nq$ (i.e. $s$ does not divide neither $l$ nor $n$), then $\LJ(u(\sigma,n))=0$.
\end{enumerate}
\end{Prop}

\begin{Rem}
It is easy to see that if $q=1$, then $s=d$. 
\end{Rem}

\subsection{The Archimedean case}\label{sec:JL Archimedean}
The above discussion can be carried out for the Archimedean case. We recall the necessary results here. For more details, see \cite{BR10}.



Let $X_{\mathbb{R}}$ be the set of unitary characters of $\mathbb{R}$. For $\chi\in X_{\mathbb{R}}$ and a positive integer $n$, we define $\chi_n:=\chi\circ \nu$ and $\chi'_n:=\chi\circ \nu_{n,\mathbb{H}}$.  

The computation of $\LJ$ on representations of the form $u(\sigma,n)$ is given as follows (\cite{BR10}, Theorem 13.8). We remind the reader that the definition of $\LJ$ in \cite{BR10} Section 4 is slightly different from what we use here. 
\begin{enumerate}
\item $\LJ(\chi_{2n})=\chi_n'$ and $\LJ(\pi(\chi_{2n},\al))=\pi(\chi'_n,\al)$ for all $\chi\in X_{\bR}$ and $\al\in (0,1/2)$.
\item If $\delta\in \mathscr{D}_2^u$ is such that $\dim \C(\delta)>1$, then
    \[
    \LJ(u(\delta,n))=(-1)^n u(\C(\delta),n), \qquad \LJ(\pi(u(\delta,k),\al))=\pi(u(\C(\delta),k),\al)
    \]
    for all $\al\in (0,1/2)$.
\item If $\delta\in \mathscr{D}_2^u$ is such that $\C(\delta)$ is a one-dimensional representation $\chi'_1$, then
    \begin{itemize}
      \item $\LJ(u(\delta,n))=\pi(\chi'_{n/2},1/2)$ and $\LJ(\pi(u(\delta,n),\al))=\pi(u(\chi'_{n/2},1/2),\al)$ if $n$ is even and $\al\in (0,1/2)$;
      \item $\LJ(u(\delta,n))=(-1)^n\chi'_{(n+1)/2}\times \chi'_{(n-1)/2}$ and $\LJ(\pi(u(\delta,n)\al))=\pi(\chi'_{(n+1)/2},\al)\times \pi(\chi'_{(n-1)/2},\al)$ if $n\neq 1$ is odd and $\al\in (0,1/2)$;
      \item $\LJ(\delta)=\chi'_1$ and $|\LJ|(\pi(\delta,\al))=\pi(\chi'_1,\al)$ for $\al\in (0,1/2)$.
    \end{itemize}
\end{enumerate}


\subsection{Transfer of unitary representations}

An irreducible unitary representation $\pi$ is written as a product of elements in $\mathcal{U}_F$. Note that $\LJ$ commutes with parabolic induction. If $\pi \in \Irr_{kd}^u$, then $\LJ(\pi)=0$ or $\LJ(\pi)$ is an irreducible unitary representation $\pi'$ of $G_{k,D}$ up to a sign. We write $\pi'=|\LJ|(\pi)$. So we have a map $|\LJ|$ from the set of $d$-compatible irreducible unitary representations of $G_{kd}$ to the set of unitary representations of $G_{k,D}$. 

As an immediate consequence, we have the following result. 

\begin{Thm}[\cite{Badulescu08,BR10}]
If $\pi$ is a $d$-compatible irreducible unitary representation of $G_{kd}$, then there exists a unique irreducible unitary representation $\pi'$ of $G_{k,D}$ and a unique sign $\varepsilon_{\pi}\in \{-1,1\}$ such that
\[
\chi_{\pi}(g)=\varepsilon_{\pi} \chi_{\pi'}(g')
\]
for all $g\in G_{nd,d}$ and $g\leftrightarrow g'$.
\end{Thm}

\begin{Rem}
The sign $\varepsilon_{\pi}$ and $\pi'$ can be compute explicitly. We refer the reader to \cite{Badulescu08} Section 3.3 for more  details. In this paper, we only need to calculate $\varepsilon$ for a special class of representations. 
\end{Rem}

\section{Quaternionic Speh representations}\label{sec:quaternionic speh}

The purpose of this section is to define a class of representation of $G_{k,D}$ with unique models. They are defined using the Jacquet-Langlands correspondence of the Speh representations. 

\subsection{The Speh representations}

In this section we recall the construction of the Speh representations and discuss some of their properties.

We now consider the local situation. Let $F$ be a local field. Let $\tau$ be an irreducible unitary generic representation of $G_k$. We attach a representation $\Speh(\tau,n)$ of $G_{kn}$ inductively as follows.

If $\tau$ is a discrete series representation, then we define $\Speh(\tau,n):=u(\tau,n)$. 

Let $\tau$ be an irreducible unitary generic representation of $G_k$. By the classification of unitary representations (\cite{Tadic86,Vogan86}), there exist discrete series representations $\tau_1,\cdots, \tau_m$ and $s_1,\cdots,s_d\in (-1/2,1/2)$ such that
\[
\tau=\tau_1\nu^{s_1}\times \cdots \times \tau_m \nu^{s_m}.
\]
The order of $1,\cdots,m$ does not change the isomorphism class of $\tau$. 
We define
\[
\Speh(\tau,n)=\Speh(\tau_1,n)\nu^{s_1} \times \cdots \times \Speh(\tau_m,n)\nu^{s_m}.
\]
By \cite{MW89} I.11, $\Speh(\tau,n)$ is irreducible and thus well-defined (i.e. the isomorphism class of $\Speh(\tau,n)$ does not depend on the order of $1,\cdots,m$).

\subsection{Digression on degenerate Whittaker models}

An important property of the Speh representations is that they are degenerate representations with unique models. We now briefly recall the theory of degenerate Whittaker models and related topics. The reader is referred to \cite{MW89,GGS17} and the recent paper \cite{GS19} for more details. 

Let $\mathfrak{g}_k$ denote the Lie algebra of $G_k$ and $\mathfrak{g}_k^{\ast}$ denotes its dual space. To every coadjoint nilpotent orbit $\mathcal{O}\subset \mathfrak{g}_k^{\ast}$ and every $\pi\in \mathrm{Rep}(G_k)$ we associate a certain generalized Whittaker quotient $\pi_{\mathcal{O}}$. Let $\mathrm{WO}(\pi)$ denote the set of all nipotent orbit $\mathcal{O}$ with $\pi_{\mathcal{O}}\neq 0$ and $\mathrm{WS}(\pi)$ denote the set of maximal orbits in $\mathrm{WO}(\pi)$ with respect to the closure ordering. We call $\mathrm{WS}(\pi)$ the Whittaker support of $\pi$. 

This can also be done in the global situation using degenerate Whittaker coefficients. We use notation $\mathcal{F}_{\mathcal{O}}(f)$ for an automorphic form $f$ to denote such degenerate Whittaker coefficients . 

We would like to introduce another nilpotent orbit attached to a representation. Denote by $\mathcal{M}(G)$ the category of admissible (finitely-generated) representations. For $\pi\in \mathcal{M}(G)$, and a nilpotent orbit $\mathcal{O}\subset \mathfrak{g}^{\ast}$, \cite{Harish-Chandra77, BV80} define a coefficient $c_{\mathcal{O}}(\pi)$ using the asymptotics of the character of $\pi$ at $1\in G$. Denote by $\mathrm{WF}(\pi)$ the set of maximal elements in the set of orbit with nonzero coefficients. For non-Archimedean $F$, this set coincides with $\mathrm{WS}(\pi)$. 


\begin{Thm}[\cite{MW87} Proposition I.11, Theorem I.16 and Corollary I.17, and \cite{Varma14}]\label{thm:relation whittaker and character expansion}
Assume that $F$ is non-Archimedean and $G$ is algebraic. Let $\pi\in \mathcal{M}(G)$. Then 
\begin{enumerate}
\item $\mathrm{WF}(\pi)=\mathrm{WS}(\pi)$.
\item For any $\mathcal{O}\in \mathrm{WF}(\pi)$, $c_{\mathcal{O}}(\pi) = \dim \pi_{\mathcal{O}}$. 
\end{enumerate}
\end{Thm}

In the Archimedean case, only partial results are known (see \cite{GGS17} Section 3.3).

\begin{Rem}
In the main body of this paper, we only use the following explicit definition of the degenerate models attached to the orbit $(k^n)_D$. This is in the group $\mathrm{GL}_{kn,D}$. The unipotent subgroup is $N_{(k^n)_D}$: 
\[
N_{(k^n)_D}=
\left\{
u \mid 
u=
\begin{pmatrix}
I_n & X_{12} & \cdots & \cdots & \ast \\
 & I_n & X_{23} & \cdots & \ast \\
&&\ddots & \cdots &\cdots\\
&&&I_n &X_{k-1,k}\\
&&&&I_n\\
\end{pmatrix} \in \GL_{kn,D}
\right\}
\]
In other words, it is the set of upper triangular unipotent $n\times n$ block matrices. We define a character 
\[
\psi_{(k^n)_D}(u)=\psi(\mathrm{tr}(X_{12} + X_{23} + \cdots + X_{k-1,k})).
\]
When $n=1$, this is the ``Whittaker'' coefficient. 

We also define 
\[
\mathrm{Wh}_{(k^n)_D}(\pi) = \Hom_{N_{(k^n)_D}}(\pi,\psi_{(k^n)_D}). 
\]
If $\mathrm{Wh}_{(k^n)_D}(\pi) \neq 0$, we say that $\pi$ supports a non-vanishing $(k,n)_D$-model. 
\end{Rem}

\begin{Rem}
Note that the above definition works when $D$ is a central simple algebra as well. Sometimes it is convenient to use this notation so we discuss it here. 

Let $D=\mathrm{M}_{m}(D')$ for a central division algebra $D'$ over $F$. Then $\GL_{n,D}=\GL_{mn,D'}$. We also have 
\[
(N_{(k^n)_D},\psi_{(k^n)_D}) = (N_{(k^{mn})_{D'}}, \psi_{(k^{mn})_{D'}}). 
\]
\end{Rem}

The Whittaker support of the Speh representations is already calculated. 

\begin{Thm}[see \cite{CFK} Section 2 and references there]\label{thm:WS of speh}
\begin{enumerate}
\item We have $\mathrm{WS}(\Speh(\tau,n))=\{(k^n)\}$. 
\item $\dim \Speh(\tau,n)_{(k^n)}=1$. In other words, the Speh representation supports unique models of degenerate type. 
\end{enumerate}
\end{Thm}

\begin{Rem}
We emphasize that the first statement consists of the the following two substatements: 
\begin{itemize}
\item For every nilptoent orbit $\mathcal{O}$ larger than $(k^n)$, $\pi_{\mathcal{O}}=0$;
\item $\pi_{(k^n)}\neq 0$. 
\end{itemize}
In some papers, it is said that the nilpotent orbit attached to $\Speh(\tau,n)$ is $(k^n)$. 
\end{Rem}

\subsection{The construction}

Speh representations over $D$ are defined as the Jacquet-Langlands transfer of the Speh representations. 

\begin{Def}
For $\tau\in \Irr^{eu}_{gen}(G_k)$ and a central division algebra $D$, we define
\[
\Speh_D(\tau,n)=|\LJ|(\Speh_F(\tau,nd)),
\]
where $|\LJ|$ is the Jacquet-Langlands correspondence from $\GL_{knd}$ to $\GL_{kn,D}$.
\end{Def}

\begin{Rem}
For a positive integer $m$, we also define
\[
\Speh_{\mathrm{M}_m(D)}(\tau,n)=\Speh_D(\tau, m n). 
\]
This is used in the global construction. 
\end{Rem}

\begin{Rem}
This construction is local-to-global compatible since the Jacquet-Langlands correspondence in \cite{Badulescu08} is local-to-global compatible.
\end{Rem}

We observe that, from Proposition \ref{prop:transfer of unitary non-arch} and the definition of $|\LJ|$, $\Speh_F(\tau,nd)$ is always $d$-compatible. Explicit constructions of $\Speh_D(\tau,n)$ can also be written done using the Aubert involution.


We would like to understand the Whittaker support of $\Speh_D(\tau,n)$. In the non-Archimedean case, this can be done using the character identity and Theorem \ref{thm:relation whittaker and character expansion}. In the Archimedean case, we also prove some partial results. We will prove further results using global methods in Section \ref{sec:Archimedean using global}. 

\subsection{Sign in the character identity}

From \cite{Badulescu08} Proposition 3.9, the representations $\Speh_D(\tau,n)$ and $\Speh_F(\tau,nd)$ satisfy a character identity
\[
\chi_{\Speh_F(\tau,nd)}(g)=\varepsilon\chi_{\Speh_D(\tau,n)}(g')
\]
with $\varepsilon\in\{\pm 1\}$ and $g\leftrightarrow g'$. We first show that $\varepsilon=1$.

\begin{Prop}\label{prop:sign in character identity}
We have
\[
\chi_{\Speh_F(\tau,nd)}(g)=\chi_{\Speh_D(\tau,n)}(g')
\]
for all $g\in \GL_{nkd}$ and $g\in \GL_{nk,D}$ such that $g\leftrightarrow g'$.
\end{Prop}

\begin{proof}

We first treat the non-Archimedean case. Recall that $\varepsilon$ is the product of $(-1)^{kn(d-1)}$ and the sign coming from the definition of $\LJ$. 

It suffices to prove the result when $\tau$ is a discrete series representation. In this case, we assume that $\tau=Z^u(\rho,l)$ for $\rho\in \scc_{p}$ where $k=pl$. We only need to use Proposition \ref{prop:transfer of unitary non-arch} (2) to calculate $\LJ(\Speh(\tau,nd))=\LJ(u(\tau,nd))$. 

Let $\sigma=Z^u(\rho,nd)$. Let $s$ be the smallest integer such that $d\mid sp$. Let $\sigma'=\mathbf{C}(\sigma)$. Then
\[
\LJ(u(\tau,nd))=\varepsilon | i'(\bar u (\sigma',l)) |,
\]
where $\varepsilon=1$ if $s$ is odd and $\varepsilon=(-1)^{\frac{ndl}{s}}$ if $s$ is even. 

We now calculate the sign in the character relation. We first consider the case when $s$ is odd. We need to show that $2\mid kn(d-1)$. We write 
\[
kn(d-1)=\dfrac{sp}{d} \dfrac{dnl}{s} (d-1)
\]
Note that both $sp/d$ and $d/s$ are integers. We now have two cases:
\begin{itemize}
\item $d-1$ is even: then the result is true. 
\item $d$ is even: since $s$ is odd and $d$ is even, $dnl/s$ must be an even integer. The result is true as well.  
\end{itemize}

We now consider the case when $s$ is even. We need to show that $2 \mid \frac{dnl}{s}-kn(d-1)$. Note that $d$ must be even. Therefore, it suffices to show that 
\[
2\mid \dfrac{dn}{s}-pn \text{ or } 2\mid \dfrac{(d-sp)n}{s}.
\]
We now write $sp=da$. We claim that $a$ must be odd. Otherwise, let $a=2a'$. Then $d\mid sp/2$. This contradicts with our choice of $s$. Therefore, $a$ is an odd integer and 
\[
\dfrac{(d-sp)n}{s}=\dfrac{d}{s}n(1-a)
\]
is an even integer. This completes the proof for the non-Archimedean case.

The Archimedean case is easier. It suffices to prove the result for the discrete series representations. Recall that for $\tau\in \mathscr{D}_k$,
\[
\chi_{\Speh(\tau,2n)}(g)= (-1)^{2kn-kn} \chi_{\LJ(\Speh(\tau,2n))}(g') = (-1)^{kn} \chi_{\LJ(\Speh(\tau,2n))}(g').
\]
By Section \ref{sec:JL Archimedean}, we have the following:
\begin{itemize}
\item $k=1$: we have that $\LJ(\Speh(\tau,2n))=(-1)^{n}|\LJ|(\Speh(\tau,2n))$.
\item $k=2$: we have that  $\LJ(\Speh(\tau,2n))=|\LJ|(\Speh(\tau,2n))$. 
\end{itemize}
This proves that if $\tau$ is a discrete series representation of either $\mathrm{GL}_1(\mathbb{R})$ or $\mathrm{GL}_2(\mathbb{R})$, then
\[
\chi_{\Speh(\tau,2n)}(g)=\chi_{\Speh_{\mathbb{H}}(\tau,n)}(g')
\]
for $g\leftrightarrow g'$. 
\end{proof}

\subsection{Whittaker supports in non-Archimedean case}\label{sec: Whittaker support non-Arch}

In this section, we determine $\mathrm{WS}(\Speh_D(\tau,n))$ and the dimension of the corresponding model in the non-Archimedean case. This is an easy consequence of the main result of \cite{Prasad00} and Proposition \ref{prop:sign in character identity}. The main idea is to use the character identity and a result of M{\oe}glin-Waldspurger \cite{MW87} and Varma \cite{Varma14}.

For an admissible representation $\pi$ of $G$, its character $\chi_{\pi}$ admits the following character expansion at identity
\[
\chi_{\pi}(\mathbf{e}(Y))=\sum_{\sco}c_{\sco}\widehat\mu_{\sco}(Y)
\]
valid for all regular semisimple $Y$ in the lie algebra $\mathfrak{g}$ such that $Y$ is close enough to $0$. Here, the sum is over the set of nilpotent orbits $\mathcal{O}$ in $\mathfrak{g}$; $\widehat\mu_{\mathcal{O}}$ is the function that represents the distribution that is the Fourier transform of the orbital integral $\mu_{\mathcal{O}}$ associated to $\mathcal{O}$; $c_{\mathcal{O}}=c_{\mathcal{O}}(\pi)\in \mathbb{C}$; and $\mathbf{e}$ is the exponential map, or some suitable substitute. 

\begin{Prop}\label{prop:Prasad result}
Suppose that $\pi\in \Irr_{kd}$ and $\pi'\in \Irr_{k,D}$. We consider that the character expansion at the identity:
\[
\chi_{\pi}=\sum_{\sco}c_{\sco}\widehat\mu_{\sco},
\]
and 
\[
\chi_{\pi'}=\sum_{\sco'}c_{\sco'}\widehat\mu_{\sco'}.
\]

If the characters satisfy the following character identity
\[
\chi_{\pi}=\varepsilon_{\pi}\chi_{\pi'},
\]
then for $\mathcal{O}\leftrightarrow \mathcal{O}'$, 
\[
c_{\mathcal{O}}=\varepsilon_{\pi} \cdot c_{\mathcal{O}'}. 
\]
\end{Prop}

\begin{proof}
The proof in \cite{Prasad00} only uses the character identity and therefore applies to our case as well. 
\end{proof}

We now prove the following result. 

\begin{Thm}\label{thm:non-Arch orbit}
We have $\mathrm{WS}(\Speh_D(\tau,n))=\{(k^n)_D\}$. Moreover,
\[
\dim \Hom_{N_{(k^n)_D}}(\Speh_D(\tau,n),\psi_{(k^n)_D})=1.
\]
\end{Thm}

\begin{proof}
This is an application of the above result and Theorem \ref{thm:relation whittaker and character expansion}. We first prove the vanishing part. Assume that there is an orbit $\mathcal{O}'$ that is greater than or not comparable with $(k^n)_D$ such that $\mathcal{O}'  \in \mathrm{WS}(\Speh_D(\tau,n))$. By Theorem \ref{thm:relation whittaker and character expansion}, $c_{\mathcal{O}'} \neq 0$. By Proposition \ref{prop:Prasad result}, $c_{\mathcal{O}} \neq 0$ for $\mathcal{O} \leftrightarrow \mathcal{O'}$. Note that $\mathcal{O}$ is either greater than or not comparable with $(k^{nd})$. Theorem \ref{thm:relation whittaker and character expansion} implies that $\mathcal{O} \in \mathrm{WS}(\Speh(\tau,nd))$. This is a contradiction. 

Similarly, we deduce that $c_{(k^n)_D}=1$ and the result now follows immediately.
\end{proof}

\begin{Rem}
The following example explains why the usual Jacquet-Langlands correspondence $\C$ does not work properly.  Assume $D$ is the unique non-split quaternion algebra over a local field $F$. Let $\mathrm{St}$ be the Steinberg representation. Then $\C(\mathrm{St})=1_{D^{\times}}$. The character relation reads 
\[
\chi_{\mathrm{St}}(g) = - \chi_{1_{D^{\times}}} (g')
\]
for $g \leftrightarrow g'$. The character expansion of $\chi_{\mathrm{St}}$ is 
\[
\chi_{\mathrm{St}} = \hat\mu_{(2)} - \hat \mu_{(1^2)}. 
\]
Since the Steinberg representation has a unique Whittaker model, $c_{(2)}=1$. But this is not related to the fact that $1_{D^{\times}}$ being one-dimensional. 

The Steinberg representation will not be considered in the domain of $|\LJ|$ for the representations we considered. Under $|\LJ|$, the trivial representation of $\GL_2(F)$ corresponds to the trivial representation of $D^\times$. In this case, the fact of $1_{\GL_2}$ being one-dimensional is related to the nilpotent orbit of $1_{D^{\times}}$. 
\end{Rem}

\subsection{Partial results in the Archimedean case}\label{sec:Arch I}

The proof in the previous section does not work well in the Archimedean case. This is because the Archimedean version of Theorem \ref{thm:relation whittaker and character expansion} is unknown (see the discussion of \cite{GS19}). However, one direction is shown in \cite{GS15}.

We first recall some basics. Let $G$ be a reductive Lie group with maximal compact subgroup $K$, and let $G_{\mathbb{C}}$ and $K_{\mathbb{C}}$ be the complexification of $G$ and $K$. Denote $\mathfrak{g}, \mathfrak{k}, \mathfrak{g}_{\mathbb{C}}$ and $\mathfrak{k}_{\mathbb{C}}$ be the Lie algebra of $G, K, G_{\mathbb{C}}$ and $K_{\mathbb{C}}$, respetively. Let $\pi$ be an irreducible admissible representation $\pi$ of $G$. We briefly recall the two invariants of cycles associated with $\pi$, one defined analytically and the other algebraically. 

For $\pi$, one has an asymptotic expansion for the character $\chi_{\pi}$ in a neighborhood of $0$ in $\mathfrak{g}$ of the form 
\[
\chi_{\pi} \sim \sum_{i=-r}^{\infty} D_i
\]
with $\{D_i\}$ being a set of tempered distributions on $\mathfrak{g}$. The asymptotic support $\mathrm{AS}(\chi_{\pi})\subset \mathfrak{g}^{\ast}$ is defined to be the union of the supports of the Fourier transforms $\hat D_i$. It is known that $\mathrm{AS}(\chi_{\pi})$ is a union of nilpotent orbits. We can view $\mathrm{AS}(\chi_{\pi})$ as a subset of $\mathfrak{g}$ by identifying $\mathfrak{g}$ and $\mathfrak{g}^{\ast}$ by the Cartan-Killing form. We set 
\[
\mathcal{N}_{\mathrm{tr}}(\pi):=\{\mathcal{O} \in \mathcal{N}: \mathcal{O} \subset \mathrm{AS}(\chi_{\pi})\},
\]
and set $\mathcal{N}_{\mathrm{tr}}^{\mathrm{max}}(\pi)$ to be the subset of maximal elements. 

For a smooth representation $\pi$ of a real reductive group $G$, one can define another invariant $\mathrm{AV}(\pi)$ -- the annihilator variety of $\pi$. It is sometimes called the associated variety of the annihilator of $\pi$. It is defined to be the set of zeros in $\fg_{\bC}^{\ast}$ of the ideal in the symmetric algebra $S(\fg_{\bC})$, which is generated by the symbols of the annihilator ideal of $\pi$ in the universal enveloping algebra. 
A result of Kostant--Rallis \cite{KR71} says that $\mathrm{AV}(\pi)$ is a finite union of nilpotent orbit $(\mathfrak{g}_{\mathbb{C}}/\mathfrak{k}_{\mathbb{C}})^{\ast}$. Identifying coadjoint orbits with adjoint orbits, and using the Sekiguchi correspondence to identify the nilpotent $K_{\mathbb{C}}$-orbits in $(\mathfrak{g}_{\mathbb{C}}/\mathfrak{k}_{\mathbb{C}})^{\ast}$ with the nilpotent $G$-orbits in $\mathfrak{g}^{\ast}$, we define the set 
\[
\mathcal{N}_{\mathrm{alg}}(\pi):=\{\mathcal{O} \in \mathcal{N} : \mathcal{O} \subset \mathrm{AV}(\pi)\}, 
\]
and let $\mathcal{N}_{\mathrm{alg}}^{\mathrm{max}}(\pi)$ be the subset of maximal elements. It follows from \cite{SV00} that $\mathcal{N}_{\mathrm{tr}}^{\mathrm{max}}(\pi)=\mathcal{N}_{\mathrm{alg}}^{\mathrm{max}}(\pi)$. 

\begin{Prop}
For $\mathcal{O'} \in \mathrm{WS}(\Speh_D(\tau,n))$, $\mathcal{O'}$ is a subset of the closure of $(k^n)_D$. 
\end{Prop}

\begin{proof}
Let us prove this result by contradiction. If $\mathcal{O'}$ is not a subset of the closure of $(k^n)_D$, then $\mathcal{O'}$ is either greater than or not comparable with $(k^n)_D$. In either case, $\mathcal{O'}$ is of the form $(k_1\cdots)$ with $k_1>k$. Recall that the definition of $\mathrm{WS}(\Speh_D(\tau,n))$ guarantees that $\mathcal{O'}$ is a maximal nilpotent orbit that support degenerate Whittaker models for $\Speh_D(\tau,n)$. 

Corollary 4 in \cite{Matumoto87}  states if $W_{\mathcal{O'}}(\pi) \neq 0$, then $\mathcal{O'} \subset \mathrm{AV}(\pi)$. Without loss of generality, we may assume that $\mathcal{O'}\in \mathcal{N}_{\mathrm{alg}}^{\mathrm{max}}(\pi)$. Then $\mathcal{O'}\in \mathcal{N}_{\mathrm{tr}}^{\mathrm{max}}(\pi)$ as well and therefore $c_{\mathcal{O}'} \neq 0$.  

The Archimedean analogue of Prasad also holds. As a consequence, from the character identity, we deduce that $c_{\mathcal{O}}\neq 0$ when $\mathcal{O}\leftrightarrow \mathcal{O}'$. Here, $c_{\mathcal{O}}$ are the coefficients appearing in the character expansion for the representation $\Speh(\tau,2n)$. Now \cite{GS15} Theorem B or \cite{GGS17} Section 3.3 implies that $\mathcal{O}\in \mathrm{WO}(\Speh(\tau,2n))$, contradicting Theorem \ref{thm:WS of speh}. This completes the proof. 
\end{proof}

\begin{Rem}
The relation between the leading coefficient and the dimension of degenerate Whittaker models is not known yet. Moreover, in order to prove the non-vanishing part, we need an extension of results in \cite{GS15} to $\GL_{k,D}$. 

In Section \ref{sec:Archimedean using global}, we use global arguments to prove some partial results towards the nonvanishing and multiplicity one. 
\end{Rem}

\subsection{A special instance}\label{sec:theta}

Regarding the multiplicity one result, we now give a proof in the minimal case using the theta correspondence. The following proof follows from a suggestion by Hang Xue. 

In this section only, let $D$ be the non-split unique quaternion algebra over a local field $F$. 

\begin{Prop}
For an irreducible admissible representation of $D^{\times}$, 
\[
\dim \Hom_{N_{(2)_{D}}}(\Speh_{D}(\tau,1),\psi_{(2)_{D}})=1. 
\]
\end{Prop}

\begin{proof}
We use the theta correspondence for the similitude pair
\[
(\mathrm{GSp}(2),\mathrm{GSO}(5,1)). 
\]
Observe that we have the following isomorphisms 
\[
\begin{aligned}
\mathrm{GL}(2) & =\mathrm{GSp}(2)\\ 
\mathrm{GSO}(5,1) & = (\mathrm{GL}_{2,D}\times \mathrm{GL}_1)/ \{(z\cdot \mathrm{Id},z^{-2}):z\in \mathrm{GL}(1) \} .\\
\end{aligned}
\]
Via these isomorphisms, 
an irreducible representation of $\mathrm{GSO}(5,1)$ is of the form $\pi \boxtimes \mu$ where $\pi$ is a representation of $\GL_{2,D}$ and $\mu$ is a square root of the central character of $\pi$. 



The theta correspondence from $\mathrm{GL}_2$ to $\mathrm{GSO}(5,1)$ for discrete series representations is given as follows
\[
\tau \mapsto \Theta(\tau)=Lg(\nu^{1/2}\C(\tau) \times  \nu^{-1/2}\C(\tau)) \boxtimes \omega_{\tau}.
\]
Here, $Lg$ denotes the Langlands quotient of the induced representation considered. Note that the above induced representation is reducible if and only if $\dim \C(\tau)>1$. We have 
\[
\Theta(\tau)=\Speh_{D}(\tau,1) \boxtimes \omega_{\tau}.
\]

To proceed, we use the result of Gomez-Zhu \cite{GZ14} to relate degenerate Whittaker models of $\tau$ and $\Theta(\tau)$. The result of Gomez-Zhu says that there is an isomorphism 
\[
\mathrm{Wh}_{\mathcal{O}}(\pi)\simeq \mathrm{Wh}_{\mathcal{O}'}(\Theta(\tau))
\] 
for two nilpotent orbits $\mathcal{O}$ and $\mathcal{O}'$ that correspond to each other under the moment map. Both $\mathrm{GL}_2$ to $\mathrm{GSO}(5,1)$ have two nilpotent orbits: the trivial one and the nontrivial one. Under the moment map, the nontrivial ones correspond to each other. As a result, we have an isomorphism between the Whittaker model of $\tau$ and $\mathrm{Wh}_{(2)_D}(\Speh(\tau,1))$. Now our result follows from the uniqueness of Whittaker models for $\mathrm{GL}(2)$.
%
\end{proof}

\section{The global Speh representations}\label{sec:global Speh}

In this section, we define the global Speh representations. 

Denote $\DS_{k}$ (resp. $\DS_{k,D}$) the set of discrete series of $G_{k}(\ba)$ (resp. $G_{k,D}(\ba)$).

\subsection{The residual spectrum of $G_k$}

We now recall the construction of the generalized Speh representations in the global setup. A theorem of M{\oe}glin-Waldspurger \cite{MW89} says that these consist of the residual spectrum of $G_k$.

Let $m$ be a positive integer and $\rho\in DS_m$ is a cuspidal representation. For a positive integer $n$, the induced representation $\prod_{i=0}^{n-1}(\nu^{\frac{n-1}{2}-i}\rho)$ has a unique constituent $\pi$ which is a discrete series (i.e. $\pi\in \DS_{mn}$). One has $\pi_v=\Speh(\tau_v,n)$. As a result, we write $\pi=\Speh(\tau,n)$. 

Discrete series of the groups $G_k(\mathbb{A})$ are all of this type, and $n$ and $\rho$ are determined by $\pi$. Moreover, $\pi$ is supercuspidal if and only if $n=1$. (By \cite{Badulescu08} Section 5.2 and \cite{BR10} Section 18, the same classification also holds for $G_{k,D}(\mathbb{A})$.)

The generalized Speh representations admit an automorphic realization. We now recall the construction. Let $F$ be a number field. Let $\tau$ be an irreducible cuspidal representation of $\GL_m(\ba)$. Let $\underline{s}=(s_1,\cdots, s_n)\in\bC^n$. We consider the following normalized induced representation
\[
\Ind_{P_{m,n}(\ba)}^{\GL_{mn}(\ba)}\tau \nu^{s_1}\otimes \cdots \otimes \tau\nu^{s_n}.
\]
Here, $P_{m,n}$ is the standard parabolic subgroup of $\GL_{mn}$ whose Levi part is $\GL_m\times \cdots \times \GL_m$ where $\GL_m$ appears $n$ times. Let $f^{(\underline{s})}$ be a flat section in the induced representation. Then we can form an Eisenstein series
\[
E(f^{(\underline{s})})(g)=\sum_{\gamma\in P_{m,n}(F)\bs \GL_{mn}(F)}f^{(\underline{s})}(\gamma g).
\]
By a result of Jacquet, this Eisenstein has a pole at the point
\[
s_1-s_2=s_2-s_3=\cdots = s_{n-1}-s_n=1,\qquad s_1+\cdots + s_n=0.
\]
The residues of $E(f^{(\underline{s})})$ at this point give an automorphic realization of $\Speh(\tau,n)$. 


As in the local case, given a nilpotent orbit, one can define degenerate Whittaker coefficients $\mathcal{F}_{\mathcal{O}}$ of an automorphic representation $\pi$. We define $\mathrm{WS}(\pi)$ to the set of maximal nilpotent orbits that support nonzero degenerate Whittaker coefficients for $\pi$. 
The set $\mathrm{WS}(\Speh(\tau,n))$ is determined by the following result.

\begin{Thm}[\cite{Ginzburg06,JL13}]
We have $\mathrm{WS}(\Speh(\tau,n))=\{(k^n)\}$. 
\end{Thm}

Sometimes, we also say that the nilpotent orbit attached to $\Speh(\tau,n)$ is $(k^n)$.

\subsection{Global Jacquet-Langlands correspondence}

Let $F$ be a number field and $D$ a central division algebra over $F$ of dimension $d^2$. Let $m\in \bn^\ast$. Set $A=\mathrm{M}_m(D)$. For each place $v$ of $F$, let $F_v$ be the completion of $F$ at $v$ and set $A_v=A\otimes F_v$. For each place $v$ of $F$, $A_v\simeq \mathrm{M}_{r_v}(D_v)$ for some positive integer $r_v$ and some central division algebra $D_v$ of dimension $d_v^2$ over $F_v$ such that $r_vd_v=md$. We will fix once and for all an isomorphism and identify these two algebras. We say that $\mathrm{M}_m(D)$ is split at a place $v$ if $d_v=1$. The set of places where $\mathrm{M}_m(C)$ is not split is a finite. For each $v$, $d_v$ divides $d$, and moreover $d$ is the smallest common multiple of the $d_v$ over all the places $v$.

Let $G_{m,D}(F)$ be the group $A^{\times}=\GL_m(D)$. For every place $v\in V$, set $(G_{m,D})_v=A_v'=\GL_{r_v}(D_v)$. In particular, if $v\notin V$, we have identified the group $\GL_{r_v}(D_v)$ and $\GL_{md}(F_v)$. For every finite place $v$ of $F$, we set $K_v=\GL_{r_v}(O_v)$, where $O_v$ is the ring of integers of $D_v$. We fix then a Haar measure $dg_v$ on $G_v'$ such that $\mathrm{vol}(K_v)=1$. For every infinite place $v$, we fix an arbitrary Haar measure $dg_v$ on $G_v'$. 

Let $\mathbb{A}$ be a ring of adeles of $F$. We consider the Haar measure $dg$ on $G'(\mathbb{A})$ which is the restricted product of the measure $dg_v$. We consider $G'(F)$ as a subgroup of $G'(\mathbb{A})$ via the diagonal embedding. 


If $\pi$ is a discrete series of $G_{kd}(\ba)$ or $G_{k,D}(\ba)$, and $v$ is a place of $F$, we denote $\pi_v$ the local component of $\pi$ at the place $v$. If $\pi$ is a discrete series of $G_{kd}(\ba)$, we say that $\pi$ is $D$-compatible if for all $v$, $\pi_v$ is $d_v$-compatible. Then $\LJ_v(\pi_v)\neq 0$ and $|\LJ|_v(\pi_v)$ is an irreducible representation of $G_v'$.

We now recall the global theorem. In the local setup, we have a map $|\LJ|:\pi\mapsto \pi'$ from the set of irreducible unitary $d$-compatible representations of $G_{kd}$ to the set of irreducible unitary representations of $G_{k,D}$.

\begin{Thm}[\cite{Badulescu08} Theorem 5.1, \cite{BR10} Theorem 18.1]
There exists a unique map $\mathbf{G}$ from the set of discrete series of $G_{k,D}(\ba)$ into the set of discrete series of $G_{kd}(\ba)$ such that $\mathbf{G}(\pi')=\pi$ implies $|\LJ|_v(\pi_v)=\pi_v'$ for all places $v\in V$ and $\pi_v=\pi'_v$ for all $v\notin V$. The map $\mathbf{G}$ is injective and onto the set of $D$-compatible discrete series of $G_{kd}(\ba)$.
\end{Thm}

We also have the multiplicity one and strong multiplicity one theorems for $G_{k,D}(\mathbb{A})$. 

\begin{Thm}[\cite{BR10} Theorem 18.1]
The multiplicity of every discrete series of $G_{k,D}(\mathbb{A})$ in the discrete spectrum is one. If two discrete series of $G_{k,D}(\mathbb{A})$ have isomorphic local components at almost every place, then they are equal. 
\end{Thm}

\begin{Prop}[\cite{BR10} Proposition 18.2]\label{thm: global def discrete series}
Let $\tau \in \DS_k$ be a cuspidal representation. Let $s_{\tau,D}$ be the smallest common multiple of $s_{\tau_v,d_v}$, $v\in V$ (see Proposition\ref{prop:transfer of unitary non-arch}). 
\begin{enumerate}
\item $\Speh(\tau,n)$ is $D$-compatible if and only if $s_{\tau, D}\mid n$. Moreover, $s_{\tau, D}\mid d$. 
\item $\mathbf{G}^{-1}(\Speh(\tau,s_{\tau, D}))=\tau'\in \DS_{m s_{\tau,D}/d,D}$ is cuspidal. In particular, $\mathbf{G}^{-1}$ sends cuspidal representations to cuspidal representations. 
\item Let $\tau'$ be a cuspidal representation of some $G_{k,D}(\ba)$. Write $\mathbf{G}(\tau')=\Speh(\tau,s_{\rho,D})$ and set $\nu_{\tau'}=\nu^{k_\tau}$. For every positive integer $m$, the induced representation 
\begin{equation}\label{eq:induced of quat speh}
\nu_{\tau'}^{(m-1)/2}\tau' \times \nu_{\tau'}^{(m-3)/2}\tau' \times \cdots \times \nu_{\tau'}^{(1-m)/2}\tau'
\end{equation}
has a unique irreducible quotient which we will denote by $\Speh'(\tau',m)$. It is a discrete series, and all discrete series are obtained for some cuspidal $\tau'$ in this way. If $\mathbf{G}(\tau')=\Speh(\tau,s_{\tau,D})$, then $\mathbf{G}(\Speh'(\tau',m))=\Speh(\tau,ms_{\tau,D})$.
\end{enumerate}
\end{Prop}

As a speical instance of this proposition, the Speh representation $\Speh(\tau,nd)$ is $D$-compatible for all central division algebra $D$ over $F$. (This is because $\Speh(\tau_v,nd)$ is already $d$-compatible and therefore $d_v$-compatible.)

We would like to note that, as in the construction of $\Speh(\tau,n)$, the representation $\Speh'(\tau',n)$ can also be constructed using residues of Eisenstein series attached to \eqref{eq:induced of quat speh}  (\cite{Badulescu08} Lemma A.5). We omit the details here.

\subsection{The global quaternionic Speh representations}

We can now define the generalized Speh representations for $G_{kn,D}(\mathbb{A})$.

\begin{Def}
For an irreducible cuspidal representation $\tau$ of $\GL_k(\mathbb{A})$ and a positive integer $n$, we define 
\[
\Speh_D(\tau,n)=\mathbf{G}^{-1}(\Speh(\tau,nd)).
\]
\end{Def}

\begin{Rem}
Sometimes we understand the construction in the following way. Consider the following family of representations: 
\[
\{\mathbf{G}^{-1}(\Speh(\tau,m))\mid  m=1,2,\cdots\}. 
\]
Then we have the following:
\begin{enumerate} 
\item the first occurance (the index for the first non-zero member) is $s_{\tau,D}$;
\item the first occurance gives a cuspidal representation;
\item$\mathbf{G}^{-1}(\Speh(\tau,m))\neq 0$  if and only if $s_{\tau,D}\mid m$. These representations can be constructed from $\mathbf{G}^{-1}(\Speh(\tau,s_{\tau,D}))$ using residues of Eisenstein series. 
\end{enumerate}
In some sense, this is similar to the theta correspondence. 
\end{Rem}

Observe that that $\Speh_D(\tau,n)$ is cuspidal only if $n=1$. The representation $\Speh_D(\tau,1)$ is cuspidal if and only if $s_{\tau,D}=d$. It is easy to construct examples such that $\Speh_D(\tau,1)$ is not cuspidal. 

For a fixed $\tau$, it is easy to find $D$ such that $\Speh_D(\tau,1)$ is cuspidal. In fact, if there exists a place $w$ such that $\tau_w$ is unramified and $D$ does not split at $w$, then $s_{\tau,D}=d$ and therefore $\Speh_D(\tau,1)$ is cuspidal. (As a result, $\Speh_D(\tau,1)$ is cuspidal for almost all $D$.)




\subsection{Fourier coefficients}

In this section, we will study Fourier coefficients of $\Speh_D(\tau,n)$. 

\begin{Thm}
For any nilpotent orbit $\mathcal{O}$ greater than $(k^n)_D$, $\mathcal{F}_{\mathcal{O}}(\phi)=0$ for all $\phi\in \Speh_D(\tau,n)$. 
\end{Thm}

\begin{proof}
This is a consequence of the local vanishing result. 
\end{proof}

We now state the global non-vanishing result. Here we consider the degenerate Whittaker coefficient
\[
\mathcal{F}_{(k^n)_D}(\phi):=\int_{N_{(k^n)_D}(F) \backslash N_{(k^n)_D}(\mathbb{A})} \phi(u)\bar\psi_{(k^n)_D}(u) \ du. 
\]

\begin{Thm}\label{thm: global non-vanishing}
There exists $\phi\in \Speh_D(\tau,n)$ such that 
\[
\mathcal{F}_{(k^n)_D}(\phi)\neq 0. 
\]
\end{Thm}

We will give a proof in the next section. Now we can say that the nilpotent orbit $(k^n)_D$ is the maximal nilpotent orbit which supports nonzero Fourier coefficients for $\Speh_D(\tau,n)$. 

\section{A global non-vanishing result}\label{sec:global non-vanishing}

The purpose of this section is to prove Theorem \ref{thm: global non-vanishing}. The strategy here is to adapt the proof of \cite{KP84} Theorem II.2.5 in our setting. We will first make some preparations. 

\subsection{Some observations}

We start with some simple observations. In this section, let $\tau$ be an irreducible cuspidal automorphic representation of $\GL_k(\ba)$. Recall that $D$ is a central division algebra over $F$. By Theorem \ref{thm: global def discrete series}, if $s_{\tau,D}=d$, then $\Speh_D(\tau,1)$ is a cuspidal representation. 

\begin{Def}
For an irreducible automorphic representation $\pi$ of $\GL_{k,D}(\mathbb{A})$, we write $\psi_{k,D}=\psi_{(k)_D}$ and define the $\psi_{k,D}$-Whittaker function of $\phi\in \pi$ as follows:
\[
W_{\phi}(g)=\int_{[N_{k,D}]} \phi(ug) \psi_{k,D}^{-1}(u) \ du 
\]
(It is straightforward to see that $W_{\phi}(g)=\mathcal{F}(\pi(g)\phi)$.) We say that $\pi$ is $D$-generic if 
\[
W_{\phi}(g) \neq 0
\]
for some $\phi \in \pi$. 
\end{Def}

\begin{Lem}\label{lem:cuspidal implies generic}
If $\pi$ is a cuspidal representation of $\GL_{k,D}(\ba)$, then it is $D$-generic. 
\end{Lem}

\begin{proof}
As in the proof of the Fourier expansion when $D=F$ (\cite{PS75,Shalika74}), we can similarly prove the following Fourier expansion for $\varphi$ in the cuspidal representation $\pi$:
\[
\phi(g)=\sum_{\gamma\in N_{k-1,D}(F)\backslash \GL_{k-1,D}(F)} W_{\phi}
\left(
\begin{pmatrix} \gamma & 0 \\ 0 & 1 \end{pmatrix} g
\right).
\]
Since $\phi\neq 0$, there exists $\gamma$ and $g$ such that $W_{\phi}
\left(
\begin{pmatrix} \gamma & 0 \\ 0 & 1 \end{pmatrix} g
\right) \neq 0$. This proves the result. 
\end{proof}

We now move to the second observation. Let us fix a local place $v_0$ and let $D_{v_0}$ be a central division algebra over $F_{v_0}$. 

\begin{Lem}
For any $n$, the local Speh representation $\Speh_{D_{v_0}}(\tau_{v_0}, n)$ admits a nonvanishing $(k,n)_{D_{v_0}}$-model.
\end{Lem} 

\begin{proof}
This is already known when $v_0$ is non-Archimedean (see Theorem \ref{thm:non-Arch orbit}). The proof here works for all places and uses global properties. 

We can choose a central division algebra $\mathbb{D}$ over $F$ such that its localization at $v_0$ is $\mathrm{M}_{n}(D_{v_0})$ and it is nonsplit for an unramified place (this holds for almost all $\mathbb{D}$). Then $s_{\tau, \mathbb{D}}=d$ and therefore, $\Speh_{\mathbb{D}}(\tau,1)$ is cuspidal. Moreover, the local component of $\Speh_{\mathbb{D}}(\tau,1)$ at $v_0$ is $\Speh_{D_{v_0}}(\tau_{v_0},n)$.  

By Theorem \ref{thm: global def discrete series}, $\Speh_{\mathbb{D}}(\tau,1)$ is $\mathbb{D}$-generic since it is cuspidal. This implies that all the local components supports the local functional. In particular, $\Speh_{D_{v_0}}(\tau_{v_0}, n)$ admits a nonvanishing $(k,n)_{D_{v_0}}$-model.
\end{proof}





\subsection{Kirillov models}

We need the following two results analogous to the Kirillov models. We only states the results here. The proofs in the non-Archimedean case will be given in Appendix \ref{app:Kirillov}. The Archimedean case will be listed as a working hypothesis and the proofs will be considered in a forthcoming article. 

Let $F$ be a local field. Let $D$ be a simple division algebra over $F$. (This gives some generality and includes degenerate cases.)

Let $P_{k,D}$ be the ``mirabolic'' subgroup of $\GL_{k,D}$ defined by 
\[
P_{k,D}:=\{g\in \GL_{k,D} : (0,\cdots,0,1)g=(0,\cdots,0,1)\}. 
\]
Let $U_{k,D}$ be the subgroup 
\[
\left\{
\begin{pmatrix} I_{k-1} & u \\ 0 & 1 \end{pmatrix}\in P_{k,D}
\right\}.
\]
The restriction of $\psi_{k,D}$ to $U_{k,D}$ is still denoted as $\psi_{k,D}$. 

Let $(D^{\times})^{\Delta}$ be the image of the diagonal embedding $D^{\times} \to \GL_{k,D}$. Let $\tilde P_{k,D}=P_{k,D} \cdot (D^{\times})^{\Delta}$. This is the standard parabolic subgroup of $\GL_{k,D}$ of type $(k-1,1)$. 

\begin{Rem}
When $D=F$ and $k=1$, $P_{k,D}$ the usual mirabolic subgroup. 
\end{Rem}

We first treat the non-Archimedean case. 

\begin{Prop}
For $\pi \in \Irr(\mathrm{GL}_{k,D})$, there is an embedding of representations of $P_{k,D}$ 
\begin{equation}\label{eq:Kirillov non-Arch}
\mathcal{J}: \ind_{N_{k,D}}^{P_{k,D}} (J_{N_{k,D},\psi_{k,D}}(\pi) \ltimes \psi_{k,D}) \hookrightarrow \pi
\end{equation}
such that for any $\lambda\in \mathrm{Wh}_{(k)_D}(\pi)$ and $f\in \ind_{N_{k,D}}^{P_{k,D}} (J_{N_{k,D},\psi_{k,D}}(\pi) \ltimes \psi_{k,D})$, 
\[
\langle \lambda, \mathcal{J}(f) \rangle = \langle \bar\lambda, f(1) \rangle. 
\]
Here $\bar \lambda: J_{N_{n,D},\psi_{n,D}}(\pi) \to \mathbb{C}$ is the map obtained from $\lambda: \pi\to \mathbb{C}$. Moreover, the embedding \eqref{eq:Kirillov non-Arch} is also $(D^{\times})^{\Delta}$-equivariant and hence $\tilde P_{k,D}$-equivariant. 
\end{Prop}

We need the fact that $ \ind_{P_{k-1,D}U_{k,D}}^{P_{k,D}}(\Phi^-(\pi)\ltimes \psi_{k,D})$ is ``cuspidal''. For the partition $(a,b)$ of $k$, we have a standard parabolic subgroup $Q_{a,b}' = M_{a,b}'U_{a,b}'$  of $G_{k,D}$. The restriction of $\psi_{k,D}$ to $U_{a,b}'$ is denoted $\psi_{k,D}$ as well. Let $P_{a,b}'$ be the stablilizer of $\psi_{k,D}$ in $M_{a,b}'$. 

\begin{Def}
Let $\pi$ be an admissible representation of $P_{k,D}$. We say that $\pi$ is $D$-cuspidal if the Jacquet module 
\[
J_{U_{a,b}'}(\pi)=0.
\]
for all $k=a+b$.
\end{Def}

\begin{Prop}\label{prop:D-cuspidal}
The representation $\ind_{N_{k,D}}^{P_{k,D}} (J_{N_{k,D},\psi_{k,D}}(\pi) \ltimes \psi_{k,D})$ is $D$-cuspidal. 
\end{Prop}

In the Archimedean case, we only consider unitarizable representations. We need the following Hypothesis and will give its proof in a separate article.  

\begin{Hypo}\label{Archimedean hypothesis}
For $\pi \in \Irr^u(\mathrm{GL}_{k,D})$, there exists a vector space $\pi_1$, an embedding 
\[
\mathcal{J}: \ind_{N_{k,D}}^{P_{k,D}} (\pi_1 \ltimes \psi_{k,D}) \hookrightarrow \pi,
\]
and an isomorphism 
\[
\mathcal{J}^{\ast}: \mathrm{Wh}_{(k)_D}(\pi) \to \pi_1^{\vee}
\]
such that
\[
\langle \lambda, \mathcal{J}(f) \rangle = \langle \mathcal{J}^{\ast}(\lambda), f(1) \rangle
\]
for all $\lambda \in \pi_1^{\vee}$ and $f\in \ind_{N_{k,D}}^{P_{k,D}} (\pi_1 \ltimes \psi_{k,D})$.
\end{Hypo}

In fact, we only need a much weaker result, as can be seen from later sections. 



\subsection{Proof of Theorem \ref{thm: global non-vanishing}: case $n=1$}



We first consider the case $n=1$. This proof here is inspired by Piatetski-Shapiro's proof of the Strong Multiplicity One Theorem. 
The proof presented here is adapted from \cite{KP84} proof of Theorem II.2.5. 

We show that $\Speh_D(\tau,1)$ is $D$-generic. We know that $\tau':=|\LJ|(\Speh(\tau, s_{\tau,D}))$ is a cuspidal representation. If $s_{\tau,D}=d$, then this representation is $D$-generic by Lemma \ref{lem:cuspidal implies generic}. We only have to treat the case $s_{\tau,D}<d$. 

The representation $\Speh_D(\tau,1)$ can be constructed from $\tau'$ using residues of Eisenstein series. With this automorphic realization, it comes with a $G_{kn,D}(F)$-invariant functional 
\[
\ell:\Speh_D(\tau,1) \to \mathbb{C}.
\] 

For simplicity, we temporarily write $\pi=\Speh_D(\tau,1)$. 
We now fix a local non-Archimedean place $v_0$. From Proposition \ref{prop:D-cuspidal}, the restriction $\Speh_{D}(\tau,1)|_{P_{k,D_{v_0}}}$ contains a $D_{v_0}$-cuspidal representation $\mathcal{K}_{v_0}$. We construct the following $P_{k,D}(\mathbb{A})$-subspace
\[
T:=\mathcal{K}_{v_0}\otimes (\otimes_{v\neq v_0}' \pi_v) \subset \pi_{v_0}\otimes (\otimes_{v\neq v_0}' \pi_v).
\]





By our construction, any $\phi\in T$ is cuspidal with respect to any unipotent subgroup of $P_{k,D}$. Therefore the function $p\mapsto \ell(\pi(p)\phi)$ can be expanded in a Fourier series 
\[
\ell(\pi(p)\phi)=\sum_{\gamma \in N_{k-1,D}(F) \backslash \GL_{k-1,D}(F)} 
W_{\phi}
\left(
\begin{pmatrix} \gamma & 0 \\ 0 & 1 \end{pmatrix} p
\right)
\]
for $p\in \tilde P_{k,D}(\mathbb{A})$. 

By the Strong Approximation Theorem, $[\tilde P_{k,D}]$ is a dense subset of $[\GL_{k,D}]$. Therefore, $\ell|_{T}\neq 0$ or in other words, $\ell(\pi(p)\phi)\neq 0$ for some $p\in P_{k,D}$ and $\varphi \in T$. This implies that  
\[
W_{\phi}
\left(
\begin{pmatrix} \gamma & 0 \\ 0 & 1 \end{pmatrix} p
\right)\neq 0
\]
for some $\gamma$. This completes the proof. 




\begin{Rem}
A similar argument shows that, for an automorphic representation of $\mathrm{GL}_n$, genericity at a local non-Archimedean place implies global genericity. 
\end{Rem}

\subsection{Proof of Theorem \ref{thm: global non-vanishing}: general case}

We now treat the case of general $n$. This is based on the so-called induction-in-stages argument and the result when $n=1$. We recall that $\Speh_D(\tau,n)$ can be constructed from $\tau'$ using residues of Eisenstein series. It is easy to see that $\Speh_D(\tau,n)=\Speh'(\tau',nd/s_{\tau,D})$. Moreover, we have the following result concerning the constant terms of $\Speh_D(\tau,n)$.

\begin{Lem}
Let $P_{(k^n)_D}=MV$ be the standard parabolic subgroup of $\GL_{kn,D}$ with Levi part 
\[
\GL_{k,D} \times \cdots \times \GL_{k,D}
\]
where $\GL_{k,D}$ appears $n$ times. For $\phi \in \Speh_D(\tau,n)$, there is a section 
\begin{equation}\label{eq:equation induction in stages}
f\in \mathrm{Ind}_{P_{(k^n)_D}(\mathbb{A})}^{\GL_{kn,D}(\mathbb{A})} (\nu_{\tau}^{\frac{d(1-n)}{2s_{\tau,D}}}\Speh_D(\tau,1) \otimes ^{\frac{d(3-n)}{2s_{\tau,D}}}\Speh_D(\tau,1) \otimes \cdots \otimes ^{\frac{d(n-1)}{2s_{\tau,D}}}\Speh_D(\tau,1) )
\end{equation}
such that the constant term $\phi_V$ of $\phi$ along $V$ is 
\begin{equation}\label{eq:identity in constant term}
\phi_{V} (g)= f(g)(I_{k,D} \times \cdots \times I_{k,D}). 
\end{equation}
Moreover, for any 
\[
\varphi \in \nu_{\tau}^{\frac{d(1-n)}{2s_{\tau,D}}}\Speh_D(\tau,1) \otimes ^{\frac{d(3-n)}{2s_{\tau,D}}}\Speh_D(\tau,1) \otimes \cdots \otimes ^{\frac{d(n-1)}{2s_{\tau,D}}}\Speh_D(\tau,1),
\]  
there is $\phi$ such that
\[
\phi_V(I_{kn,D})=\varphi. 
\] 
\end{Lem}

\begin{proof}
The proof of \cite{JL13} Lemma 4.1 can be easily adapted to this case. The last part is straightforward. 
\end{proof}

We want to show that $\Speh_D(\tau,n)$ has a non-vanishing $(k^n)_D$ Fourier coefficient. We now introduce a slight different Fourier coefficient. 

We define the following character $\psi_{(k^n)_D}'$ on $N_{kn,D}$ as follows. We first write $u \in N_{kn,D}$ as the product of 
\[
u'=\mathrm{diag}(u_1, \cdots, u_n)\in M=\GL_{k,D} \times \cdots \times \GL_{k,D}, \qquad u_i \in N_{k,D}
\]
and $u''\in V$. We define 
\[
\psi_{(k^n)_D}(u) = \psi_{(k)_D}(u_1+\cdots + u_n).
\]
In other words, $\psi'$ is the extension of the $(k)_D \times \cdots \times (k)_D$-coefficients of the Levi part. We then set
\[
\mathcal{F}'_{(k^n)_D}(\phi)= \int_{[N_{kn,D}]} \phi(u) \bar\psi_{(k^n)_D}(u)  \ du. 
\]
\begin{Lem}
The Fourier coefficient $\mathcal{F}_{(k^n)_D}(\phi) \neq 0$ for some $\phi \in \Speh_D(\tau,n)$ if and only if $\mathcal{F}'_{(k^n)_D}(\phi) \neq 0$ for some $\phi \in \Speh_D(\tau,n)$.
\end{Lem}

\begin{proof}
This is a special case of \cite{GGS} Theorem 8.2.1. 
\end{proof}

We now show that $\mathcal{F}'_{(k^n)_D}(\phi) \neq 0$ for some $\phi \in \Speh_D(\tau,n)$. Note that this Fourier coefficient can be written as the composition of a constant term along $V$ and a Fourier coefficient for a Levi subgroup $M$. In other words, 
\[
\mathcal{F}'_{(k^n)_D}(\phi) = \int_{[N_{kn,D}\cap M]} \int_{[V]} \phi(u'' u') \ du''  \bar\psi_{kn,D}'(u') \ du'.
\] 
It is enough to show that 
\[
\int_{[N_{kn,D}\cap M]}  f(u')(I_{k,D} \times \cdots \times I_{k,D}) \bar\psi_{kn,D}'(u') \ du'
\]
is nonzero for some $f$ in \eqref{eq:equation induction in stages}. This is a $(k)_D \times \cdots \times (k)_D$ for the representation 
\[
\nu_{\tau}^{\frac{d(1-n)}{2s_{\tau,D}}}\Speh_D(\tau,1) \otimes ^{\frac{d(3-n)}{2s_{\tau,D}}}\Speh_D(\tau,1) \otimes \cdots \otimes ^{\frac{d(n-1)}{2s_{\tau,D}}}\Speh_D(\tau,1).
\] 
This is nonzero for some choice $\varphi$ from the base case $n=1$. This completes the proof of Theorem \ref{thm: global non-vanishing}. 

\subsection{An Archimedean result}\label{sec:Archimedean using global}

In this section, we treat the Archimedean case. We first prove that $\Speh_D(\tau,n)$ has a unique $(k,n)_D$-model when $\tau$ appears as the local component of a global cuspidal representation. This proof here is also inspired by Piatetski-Shapiro's proof of the Strong Multiplicity One Theorem. See also \cite{KP84} proof of Theorem II.2.5. 
\begin{Thm}
Let $\tau_{\infty}$ be an irreducible unitary generic representation of $\GL_k(\mathbb{R})$, which can be realized as the local component of an irreducible cuspidal representation of $\GL_k(\mathbb{A})$. Then under Hypothesis \ref{Archimedean hypothesis},
\[
\dim \Hom_{N_{(k,n)_{\mathbb{H}}}}(\Speh_{\mathbb{H}}(\tau_{\infty},n),\psi_{(k,n)_{\mathbb{H}}})=1. 
\]
\end{Thm}


\begin{proof}
Let $\tau$ be a cuspidal representation of $\GL_k(\mathbb{A})$ which has $\tau_{\infty}$ as its local compoent $\tau_{v_1}$ at a real place. We first choose a central division $D$ of dimension $(2n)^2$ over $F$ such that $D_{v_1}=\mathrm{M}_n(\mathbb{H})$ and is non-split (at least) at another non-Archimedean place where $\tau$ is unramified. We also assume that $D$ splits over all other Archimedean place other than $v_1$. With this choice, $\Speh_D(\tau,1)$ is a cuspidal representation of $\GL_{k,D}(\mathbb{A})$ and $\Speh_{\mathbb{H}}(\tau_{\infty},n)$ appears as the local component of $\Speh_D(\tau,1)$ at $v_1$.  


We now factor 
\[
\GL_{k,D}(\mathbb{A})= \GL_{k,D}(F_{v_1}) \times \GL_{k,D}(\mathbb{A}_{v\neq v_1}).  
\]
Again, for ease of notations, let us write $\pi=\Speh_D(\tau,1)$. We can decompose the representation $\Speh_D(\tau,1)$ as 
\[
\pi=\Speh_D(\tau,1) = \pi_{v_1} \otimes \pi_{v\neq v_1} . 
\]

Since $\Speh_D(\tau,1)$ is a cuspidal representation, we now have the Fourier expansion:
\[
\ell(x\otimes y)=
\sum_{\gamma \in N_{k-1,D}(F) \backslash \GL_{k-1,D}(F)} 
W_{x\otimes y}
\left(
\begin{pmatrix} \gamma & 0 \\ 0 & 1 \end{pmatrix} 
\right),
\]
where $x\otimes y \in   \pi_{v_1} \otimes \pi_{v\neq v_1} $.

For ease of notations, we temporarily write $V$ to be representation space of $\Speh_D(\tau,1)$ and write $V=V_{v_1}\otimes V_{v\neq v_1}$. 
We already know that the $(k,1)_{D_v}$-models for $V_v$ is unique for every place $v\neq v_1$. Thus we can choose $\lambda_{v\neq v_1}$ to be a generator of $\mathrm{Wh}_{(k)_D}(V_{v\neq v_1})$.  Then there exists a $\lambda_{v_1} \in \mathrm{Wh}_{(k)_D}(V_{v_1})$ such that for all $x'\in V_{v_1}$ and $y'\in V_{v\neq v_1}$ one has 
\[
W_{x'\otimes y'}(1)=\lambda_{v\neq v_1}(x') \cdot \lambda_{v_1}(y'). 
\]

Recall that by Hypothesis \ref{Archimedean hypothesis},  there is a representation $\pi_1$ of the trivial group with a $P_{k,\mathbb{H}}$
\[
\mathcal{J}: \ind_{N_{k,\mathbb{H}}}^{P_{k,\mathbb{H}}}(\pi_1\ltimes \psi_{k,D})\to V_{v_1}
\]
and an isomorphism
\[
\mathcal{J}^{\ast}: \mathrm{Wh}_{(k)_D}(V_{v_1}) \to \pi_1^{\vee}
\]
such that 
\[
\langle \mathcal{J}(f), f^{\ast} \rangle = \langle f(1), \mathcal{J}^{\ast} (f^{\ast}) \rangle
\]
for all $f\in \ind_{N_{k,\mathbb{H}}}^{P_{k,\mathbb{H}}}(\pi_1\ltimes \psi_{k,D})$ and $f^{\ast} \in \pi_1^{\vee}$.

Assume that $\dim \mathrm{Wh}_{(k)_D}(V_{v_1})=\dim \pi_1^{\vee}>1$. By consider the kernel of $\mathcal{J}^{\ast}(\lambda_{v_1})$, we see that there exists a $P_{k,\mathbb{H}}$-subspace $U_{v_1}$ of $V_{v_1}$ such that 
\[
\lambda_{v_1}(u)=0, \text{ for all }u\in U_{v_1}.
\]

We now set $U_{\mathbb{A}}:= V_{v\neq v_1} \otimes U_{v_1}$. This is a $P_{k,D}(F_{v_1}) \times \GL_{k,D}(\mathbb{A}_{v\neq v_1})$-subspace. From the Fourier expansion we also have 
\[
\ell(u)=0, \text{ for all } u \in U_{\mathbb{A}}. 
\]

For $u\in U_{\mathbb{A}}$, we consider the following function on $\GL_{k,D}(F)\backslash \GL_{k,D}(\mathbb{A})$:
\[
g\mapsto \ell(\pi(g)u).
\]
Let us fix $u$ and consider its set of zeros $Y$. Then we know the following:
\begin{itemize}
\item $P_{k,D}(F_{v_1}) \times \GL_{k,D}(\mathbb{A}_{v\neq v_1})\subset Y$;
\item $Y$ is left-invariant under $\GL_{k,D}(F)$ and the center $Z(\GL_{k,D}(\mathbb{A}))$.
\end{itemize} 
By the Strong Approximation Theorem, we see that $Y$ contains $\mathrm{SL}_{k,D}(\mathbb{A})$ as it contains $\mathrm{SL}_{k,D}(F) \cdot \SL_{k,D}(F_S)$ as a dense subset for a sufficiently large set of places $S$. 

We can now deduce that $Y$ contains $P_{k,D}(F_{v_1})\mathrm{SL}_{k,D}(F_{v_1}) \times \GL_{k,D}(\mathbb{A}_{v\neq v_1})$. Now it is easy to see that 
\[
Z(\mathrm{GL_{k,D}}(F_{v_1}))P_{k,D}(F_{v_1})\mathrm{SL}_{k,D}(F_{v_1}) = \GL_{k,D}(F_{v_1}).
\]
This implies that $Y=\GL_{k,D}(\mathbb{A})$. 

From this, we deduce that for $u\in U_{\mathbb{A}}$, one has $\ell(\pi(g)u)=0$ for all $g\in \GL_{k,D}(\mathbb{A})$. Thus $U_{\mathbb{A}}$ generates a proper $\GL_{k,D}(\mathbb{A})$-subspace of $V$. This is impossible since $V$ is irreducible. This completes the proof. 
\end{proof}

\begin{Rem}
It is easy to see that the same proof will prove the following statement: if the Speh representation has unique $(k,n)_{D_v}$-functional at every non-Archimedean local place, then so does every Archimedean place. 
\end{Rem}

We now state another working hypothesis. 
\begin{Hypo}\label{hypo:induction}
For $i=1,2$, let $\pi_i$ be an irreducible representation of $\GL_{k_i n,\mathbb{H}}$ such that $\mathrm{WS}(\pi_i)=\{(k_i^n)_{\mathbb{H}}\}$ and 
\[
\dim \mathrm{Wh}_{(k_i^n)_{\mathbb{H}}}(\pi_i)\leq 1. 
\] 
Set $k=k_1+k_2$. Then 
\[
\dim \mathrm{Wh}_{(k^n)_{\mathbb{H}}}(\pi_1\times \pi_2)\leq 1. 
\]
\end{Hypo}

As a corollary, we have the following. 

\begin{Cor}
Assuming Hypothesis \ref{Archimedean hypothesis} and \ref{hypo:induction}, we have that 
\[
\dim \mathrm{Wh}_{(k^n)_{\mathbb{H}}} (\Speh_{\mathbb{H}}(\tau,n)) \leq 1. 
\]
\end{Cor}

\appendix

\section{Kirillov models}\label{app:Kirillov}


In this section, we prove some results regarding the representation theory of the local groups. Recall that a important result in the representation theory of $\GL_k(F)$ is that every generic representation admits a Kirillov model (\cite{BZ77,JS81a}). In this section, we would like to prove similar results for representations of $G_{k,D}$. 

\subsection{Basic setup}

Let $F$ be a local field. To include the ``degenerate'' case, here we allow $D$ to be a central simple algebra (instead of a central division algebra) over $F$. 

Let $P_{k,D}$ be the ``mirabolic'' subgroup of $G_{k,D}$ defined by 
\[
P_{k,D}:=\{g\in G_{k,D} : (0,\cdots,0,1)g=(0,\cdots,0,1)\}. 
\]
Let $U_{k,D}$ be the subgroup 
\[
U_{k,D}:=
\left\{
u \mid 
u=\begin{pmatrix} I_{k-1} & x \\ 0 & 1 \end{pmatrix}\in P_{k,D}
\right\}.
\]
The restriction of $\psi_{k,D}$ to $U_{k,D}$ is still denoted as $\psi_{k,D}$. 

\subsection{The non-Archimedean case}

We first treat the non-Archimedean case since the argument is much easier. We introduce several functors. For an algebraic representation $\pi$ of $P_{k,D}$, we define a functor using the twisted Jacquet module
\[
\Phi^-: \mathrm{Alg}(P_{k,D}) \to \mathrm{Alg}(P_{k-1,D}), \qquad \pi\mapsto J_{U_{k,D},\psi_{k,D}}(\pi). 
\]
For an algebraic representation $\pi$ of $P_{k-1,D}$, we define 
\[
\Phi^{+}:  \mathrm{Alg}(P_{k-1,D}) \to \mathrm{Alg}(P_{k,D}), \qquad \pi\mapsto \ind_{P_{k-1,D}U_{k,D}}^{P_{k,D}}(\pi\ltimes \psi_{k,D}). 
\]

\begin{Lem}
For an algebraic representation $\pi$ of $P_{k,D}$, there is a natural homomorphism 
\[
j:\Phi^+\Phi^-(\pi) \to \pi. 
\]
\end{Lem}

\begin{proof}
The argument in the proof of \cite{BZ76} Proposition 5.12(b) works here as well. We describe it briefly here. 

The representation $\pi|_{U_{k,D}}$ can be viewed as a representation of the Hecke algebra $(C_c^{\infty}(U_{k,D}),\ast)$, where $\ast$ is the convolution.  Using the Fourier transform, $\pi|_{U_{k,D}}$ becomes a representation of $(C_c^{\infty}(\widehat{U}_{k,D}), \cdot)$, where $\cdot$ denotes the pointwise multiplication. As a result, we view $\pi$ as an $l$-sheaf $\mathcal{F}^{\pi}$ on $\widehat U_{k,D}$. 

The action of $G_{k-1,D}$ acts on $\widehat U_{k,D}$ with only one open dense orbit, and $\psi_{k,D}$ is a representative for this orbit. We now restrict the sheaf $\mathcal{F}^{\pi}$ to the open orbit. Since the stabilizer of $\psi_{k,D}$ in $P_{k,D}$ is $P_{k-1,D}U_{k,D}$, the restriction sheaf corresponds to an induced representation from $P_{k-1,D}U_{k,D}$ to $P_{k,D}$. The inducing data is given by the stalk of this sheaf at $\psi_{k,D}$, which is $\Phi^{-}(\pi)$. This completes the proof. 
\end{proof}

\begin{Lem}\label{lem:non-arch Kirillov induction step}
Any $\lambda \in \mathrm{Wh}_{(k)_D}(\pi)$ factors through $\bar\lambda: \Phi^{-}(\pi) \to \mathbb{C}$. Moreover, 
\[
\langle \lambda, j(f) \rangle = \langle \bar \lambda,f(1) \rangle
\]
for any $f\in  \ind_{P_{k-1,D}U_{k,D}}^{P_{k,D}}(\Phi^-(\pi)\ltimes \psi_{k,D})$.
\end{Lem}

\begin{proof}
This is a consequence of the above Lemma.
\end{proof}

Similarly, any $\lambda\in \mathrm{Wh}_{(k)_D}$ factors through $\mathcal{J}^{\ast}: J_{N_{k,D},\psi_{k,D}}(\pi) \to \mathbb{C}$. 

\begin{Cor}
We have a natural homomorphism 
\[
\mathcal{J}: \ind_{N_{k,D}}^{P_{k,D}} (J_{N_{k,D},\psi_{k,D}}(\pi)\ltimes \psi_{k,D}) \to \pi
\]
such that 
\[
\langle \lambda, \mathcal{J}(f) \rangle = \langle \mathcal{J}^{\ast}(\lambda), f(1) \ra
\]
for all $f\in \ind_{N_{k,D}}^{P_{k,D}} (J_{N_{k,D},\psi_{k,D}}(\pi)\ltimes \psi_{k,D})$. 
\end{Cor}

\begin{proof}
This is a consequence of Lemma \ref{lem:non-arch Kirillov induction step} and induction. 
\end{proof}

Finally, we show that $ \ind_{P_{k-1,D}U_{k,D}}^{P_{k,D}}(\Phi^-(\pi)\ltimes \psi_{k,D})$ is ``cuspidal''. For the partition $(a,b)$ of $k$, we have a standard parabolic subgroup $Q_{a,b}' = M_{a,b}'U_{a,b}'$  of $G_{k,D}$. The restriction of $\psi_{k,D}$ to $U_{a,b}'$ is denoted $\psi_{k,D}$ as well. Let $P_{a,b}'$ be the stablilizer of $\psi_{k,D}$ in $M_{a,b}'$. 

\begin{Prop}
Let $\tau$ be a smooth representation of $P_{a,b}'$. Then $\ind_{P_{a,b}'U_{a,b}}^{P_{k,D}} (\tau\ltimes \psi_{k,D})$ is $D$-cuspidal, in the following sense: for any partition $(a,b)$ of $k$, the Jacquet module 
\[
J_{U_{a,b}'}(\ind_{P_{a,b}'U_{a,b}'}^{P_{k,D}} (\tau\ltimes \psi_{k,D}))=0.
\]
\end{Prop}

\begin{proof}
Let $X=P_{a,b}'U_{a,b}' \backslash P_{k,D}$. Then the representation $\ind_{P_{a,b}'U_{a,b}'}^{P_{k,D}} (\tau\ltimes \psi_{k,D})$ corresponds to an $l$-sheaf $\mathcal{F}$ on $X$. To prove the result, it suffices to show there is no $U_{a,b}'$-equivariant functional $\mathcal{F}(X) \to \mathbb{C}$. 

We apply the Bernstein's localization principle to prove this statement. Note that the action of $U_{a,b}'$ on $X$ is trivial. Thus it is enough to show that there is no $U_{a,b}'$-equivariant functional on each stalk of $\mathcal{F}$. Notice that the action of $U_{a,b}'$ on the stalk $\mathcal{F}_x$ is through a conjugation of $\psi_{k,D}$, which is non-trivial. Thus, it is impossible for the stalks to have $U_{a,b}'$-equivariant functionals. This proves the result. 
\end{proof}

\begin{Cor}
The representation $\ind_{N_{k,D}}^{P_{k,D}} (J_{N_{k,D},\psi_{k,D}}(\pi) \ltimes \psi_{k,D})$ is $D$-cuspidal.
\end{Cor}

\begin{proof}
By induction-in-stages, we can write 
\[
\ind_{N_{k,D}}^{P_{k,D}} (J_{N_{k,D},\psi_{k,D}}(\pi) \ltimes \psi_{k,D})=\ind_{P_{a,b}'U_{a,b}}^{P_{k,D}} ((\ind_{N_{k,D}}^{P_{a,b}'}J_{N_{k,D},\psi_{k,D}}(\pi) \ltimes\psi_{k,D})\ltimes \psi_{k,D}).
\]
Now the result follows from the above result. 
\end{proof}



\end{document}